\documentclass[11pt]{article}
\usepackage[english]{babel}
\usepackage[letterpaper,top=2cm,bottom=2cm,left=3cm,right=3cm,marginparwidth=1.75cm]{geometry}

\usepackage{amssymb,amsmath,accents}
\usepackage{amscd}
\usepackage{amsfonts,amsthm,mathrsfs,mathtools}
\usepackage{enumitem}
\usepackage{setspace}
\usepackage{graphicx, color}
\usepackage{soul}
\usepackage{todonotes}
\usepackage{easyReview}
\numberwithin{equation}{section}
\newtheorem{thm}{Theorem}[section]
\newtheorem{cor}[thm]{Corollary}
\newtheorem{lem}[thm]{Lemma}
\newtheorem{prop}[thm]{Proposition}

\newtheorem{remark}[thm]{Remark}

\title{Asymptotically self-similar graph-like solutions \\ to a multi-dimensional surface diffusion flow equation \\ under contact angle and no-flux boundary conditions\\
\vspace{1em}
\small{In the memory of Professor Louis Nirenberg -- \\
a founding father of modern PDE theory}}
\author{Yoshikazu Giga and Sho Katayama}
\date{}

\begin{document}
\maketitle

\begin{abstract}
This paper studies Mullins' model of thermal grooving which consists of a surface diffusion flow equation with contact angle and no-flux boundary conditions.
 We consider this problem in a multi-dimensional half space and prove that if the slope of the initial data is close to that consistent with the contact angle, then there exists a unique global-in-time solution.
 In particular, we show the existence of a self-similar solution for a given behavior at the space infinity.
 We also show that our global solution converges to a self-similar solution as the time tends to infinity if the initial data is asymptotically homogeneous at the space infinity.
 No assumption on the size of the contact angle is imposed.
\end{abstract}

%%%%%%%%%%%%%%
%\noindent\textbf{AMS subject classifications 2020.}\quad XXXXX, YYYYY

%\keywords{aaa, bbb}
%%%%%%%%%%%%%%

%%%%%%%%%%%%%%%%%%%%%%%%%%%%%%%%%%%%%%%%%%%%%%%%%
\section{Introduction} \label{SI} % Section 1
% 原稿 2025/11/30、1/4
We consider a surface diffusion flow equation of an evolving family $\{\Gamma(t)\}_{t\ge0}$ of hypersurfaces in a half space $\mathbb{R}_+^{N+1}=\mathbb{R}_+^N\times\mathbb{R}$, where
\[
    \mathbb{R}_+^N = \left\{ x=(x_1, \ldots, x_N) \bigm|
    x_N > 0 \right\},
\]
for $N=1,2,\ldots$ i.e., $N\in\mathbb{Z}_{\ge1}$ with a contact angle boundary condition together with the no-flux boundary condition.
 It is of the form
\begin{align}
    V &= -\Delta_\Gamma H &\hspace{-6em}
    &\text{on}\quad
    \Gamma(t) \subset \mathbb{R}_+^{N+1}, \label{E1} \\
    \nu\cdot\mathbf{n} &= \cos \left( \frac{\pi}{2}-\theta\right) &\hspace{-6em}
    &\text{on}\quad
    \bar{\Gamma}(t) \cap \partial\mathbb{R}_+^{N+1}, \label{E2} \\
    \nu\cdot\nabla_\Gamma H &= 0 &\hspace{-6em}
    &\text{on}\quad
    \bar{\Gamma}(t) \cap \partial\mathbb{R}_+^{N+1}, \label{E3}
\end{align}
for $t>0$.
 Here $H$ denotes the ($N$ times) mean curvature of $\Gamma(t)$ in the direction of the unit normal $\mathbf{n}$ of $\Gamma(t)$, and $\Delta_\Gamma=\operatorname{div}_\Gamma\nabla_\Gamma$ denotes the Laplace Beltrami operator, where $\nabla_\Gamma$ denotes the surface gradient while $\operatorname{div}_\Gamma$ denotes the surface divergence;
 see e.g.\ \cite{G06} and \refeq{SA}; $V$ denotes the normal velocity of $\Gamma(t)$ in the direction of $\mathbf{n}$.
 The equation \eqref{E1} is called the surface diffusion flow equation.
 The boundary condition \eqref{E2} is called a contact angle condition with angle $\frac{\pi}{2}-\theta$.
 Here $\theta$ is given constant with $|\theta|<\pi/2$ and $\nu=(0,\ldots,0,-1,0)$ is the outward unit normal of $\partial\mathbb{R}_+^{N+1}$ and $\nu\cdot\mathbf{n}$ denotes the inner product of $\nu$ and $\mathbf{n}$.
 The condition \eqref{E3} is often called the no-flux condition.
 
% 原稿 2025/11/30、2/4
The goal of this paper is to study the large time behavior of a solution when $\Gamma(0)$ is given as the graph of a function $u_0$ on $\mathbb{R}_+^N$ (a graph-like surface) as well as the existence of a unique global-in-time classical solution.
 Our results are summarized as follows.
\begin{enumerate}
\item[(i)] (Global existence)
 If $\nabla u_0=(D_1 u_0,\ldots,D_N u_0)$ is close to $(\tan\theta)e_N$ with $e_N=(0,\ldots,0,1)$ (at least in $L^\infty$ sense), there exists a unique global-in-time (still graph-like) solution
\[
    \Gamma(t) = \left\{ \left(x,u(x,t) \right) \mid
    x \in \mathbb{R}_+^N \right\}
\]
to \eqref{E1}--\eqref{E3} such that $\nabla u$ is still close to $(\tan\theta)e_N$.
 (We do not assume that $\tan\theta$ is small, unlike \cite{AG14} and \cite{AK25}.)
\item[(i\hspace{-0.1em}i)] (Self-similar solution)
 Assume that $u_0$ is homogeneous in the sense that $u_0(\sigma x)=\sigma u_0(x)$ for all $x\in\mathbb{R}_+^N$ and $\sigma>0$ and close to $(\tan\theta)e_N$.
 Then, problem \eqref{E1}, \eqref{E2} possesses a unique self-similar solution
\[
    \Gamma(t) = \left\{ \left(x,u(x,t) \right) \mid
    x \in \mathbb{R}_+^N \right\},
\]
i.e., a solution satisfying $u^\sigma=u$ in $\mathbb{R}_+^{N+1}\times(0,\infty)$ for all $\sigma>0$, where $u^\sigma(x,t)=\sigma^{-1/4}u(\sigma^{1/4}x,\sigma t)$, such that $\nabla u$ is close to $(\tan\theta)e_N$ with $u|_{t=0}=u_0$.
 Note that $u=(\tan\theta)e_N$ is always a trivial self-similar solution.
 Depending on the behavior at space infinity (characterized by $u_0$), there exist many nontrivial self-similar solutions.
 For example, if $\tan\theta$ is not zero but small, we may take $u_0\equiv0$ to get a non-trivial (spatially bounded) self-similar solution.
\item[(i\hspace{-0.1em}i\hspace{-0.1em}i)] (Stability)
 If initial data is close to a homogeneous function at the space infinity, the solution to \eqref{E1}--\eqref{E3} converges to a self-similar solution as $t\to\infty$.
\end{enumerate}

The global existence result (i) is known when $\Gamma(t)$ is a curve, i.e., $N=1$, by T.~Asai and Y.~Kohsaka \cite{AK25}, where $\tan\theta$ is assumed to be small and $u_0$ is bounded.
 This implies (i\hspace{-0.1em}i) for small $\tan\theta$ by taking $u_0\equiv0$.
% We do not impose smallness of $\tan\theta$.
 If the no-flux condition \eqref{E3} is linearized around $u_{xx}=0$ for $N=1$, results similar to
% 原稿 2025/11/30、3/4
 (i)--(i\hspace{-0.1em}i\hspace{-0.1em}i) were proved by T.~Asai and the first author \cite{AG14}, when $\tan\theta$ is small and $u_0$ is at least bounded.
 A general strategy to construct a self-similar solution used in these papers and also in the present paper is to solve the problem globally in time for homogeneous initial data $u_0$ i.e., $u_0$ satisfies $u_0^\sigma(x)=u_0(x)$ for all $\sigma>0$.
 This idea was first introduced by the first author and T.~Miyakawa \cite{GM89} for the Navier--Stokes equations and developed for many other equations by T.~Cazenave and F.~B.~Weissler \cite{CW98}.
 For the surface diffusion flow equations, a graph-like self-similar solution on the whole space $\mathbb{R}^N$ was first constructed by H.~Koch and T.~Lamm \cite{KL12}.
 The stability (i\hspace{-0.1em}i\hspace{-0.1em}i) of a self-similar solution was already noted in \cite{AG14}.
 However, in order to prove the uniform convergence in $\mathbb{R}_+^N$, i.e., convergence in $L^\infty(\mathbb{R}_+^N)$, one needs to control behavior near space infinity of $u^\sigma$ uniformly for $\sigma>1$, for example, by showing the equi-decay property of $u^\sigma$, as studied in a book \cite{GGS10} by the first author with M.-H.~Giga and J.~Saal for the vorticity equations.
 For a graph-like self-similar solution on $\mathbb{R}^N$ to the surface diffusion flow equation, the stability in $L^\infty(\mathbb{R}^N)$ sense was established by H.~Du and N.~K.~Yip \cite{DY23} by controlling behavior at the space infinity.
 In our present paper, we do not touch stability in $L^\infty(\mathbb{R}_+^N)$.
 We also note that
% 原稿 2025/11/30、4/4
it is not trivial to say that if $\Gamma(0)$ is graph-like, so is $\Gamma(t)$.
 In fact, for the surface diffusion flow equation in the whole $\mathbb{R}^1$, a graph-like solution may lose the graph-like property in a finite time as shown by C.~M.~Elliott and S.~Maier--Paape \cite{EM01}.

% 原稿 2025/12/9、1追-1/4
To explain our strategy, let us write down the equations \eqref{E1}--\eqref{E3} when $\Gamma(t)$ is given as the graph of a function $u=u(x,t)$, i.e., 
\[
    \Gamma(t) = \left\{ \left(x,u(x,t) \right) \mid
    x \in \mathbb{R}_+^N \right\}.
\]
The details are given in \refeq{SA} for the reader's convenience.
 Let $\mathbf{n}$ be the upward unit normal vector field of $\Gamma(t)$.
 Its explicit form is
\[
    \mathbf{n} = \frac{(-\nabla u,1)}{\omega}, \quad
    \omega = \sqrt{1 + |\nabla u|^2}.
\]
Then the contact angle condition \eqref{E2} can be written as
\[
    D_N u - \omega \sin \theta = 0.
\]
If we use the tangential gradient $\nabla'=(D_1,\ldots,D_{N-1})$, then this can be rewritten as
\begin{equation} \label{EG2}
    D_N u - \gamma \sqrt{1+|\nabla' u|^2} = 0, \quad
    \gamma = \tan \theta.
\end{equation}
We next observe that
\[
    \nabla_\Gamma H = \left( P(\nabla u) \nabla H, \frac{\nabla u \cdot \nabla H}{\omega^2} \right), \quad
    P(\nabla u) = I_N - \frac{\nabla u\otimes\nabla u}{1+|\nabla u|^2},
\]
where $I_N$ denotes the $N\times N$ identity matrix; see Proposition~\ref{PA2}.
 Then the no-flux condition \eqref{E3} can be written as
\begin{equation} \label{EG3}
    \left( P(\nabla u) \nabla H \right)_N = 0
    \quad\text{or}\quad
     \left( \omega P(\nabla u) \nabla H \right)_N = 0,
\end{equation}
where $(w)_N$ denotes the $N$th component of a vector $w$.
 Since the upward normal velocity $V$ equals
\[
    V = \frac{u_t}{\omega}, \quad
    u_t = \frac{\partial u}{\partial t},
\]
the equation \eqref{E1} can be written as
\[
    \frac{u_t}{\omega} + \frac1\omega \operatorname{div}\left( \omega P(\nabla u) \nabla H \right) = 0
\]
or
\begin{equation} \label{EG1}
    u_t + \operatorname{div} \left( \omega P(\nabla u) \nabla H \right) = 0
\end{equation}
if we notice that $\Delta_\Gamma f=\omega^{-1}\operatorname{div}\left(\omega P(\nabla u)\nabla f\right)$; see Proposition~\ref{PA3}.
 The mean curvature $H$ can be written as
\[
    H = \operatorname{tr} \left( \frac{P(\nabla u) \nabla^2 u}{\omega} \right)
    = \operatorname{div} \left(\frac{\nabla u}{\omega} \right),
\]
where $\operatorname{tr}$ denotes the trace; see Proposition~\ref{PA4}.
% 原稿 2025/12/9、1追-2/4
 Gathering \eqref{EG2}--\eqref{EG1}, we now obtain equations for $u$ equivalent to \eqref{E1}--\eqref{E3}.
 Its explicit form is
\begin{align}
    u_t + \operatorname{div} \left( \sqrt{1+|\nabla u|^2} P(\nabla u)
    \nabla \operatorname{tr}\left( \frac{P(\nabla u)}{\sqrt{1+|\nabla u|^2}} \nabla^2 u\right) \right)
    &= 0 \quad\text{in}\quad
    \mathbb{R}_+^N \times (0,T), \label{EU1} \\
    D_N u - \gamma \left( 1+|\nabla'u|^2 \right)^{1/2}
    &= 0 \quad\text{in}\quad
    \mathbb{R}_+^N \times (0,T), \label{EU2} \\
    \left( \sqrt{1+|\nabla u|^2} P(\nabla u)
    \nabla \operatorname{tr} \left( \frac{P(\nabla u)}{\sqrt{1+|\nabla u|^2}} \nabla^2 u \right) \right)_N
    &= 0 \quad\text{in}\quad
    \partial\mathbb{R}_+^N \times (0,T). \label{EU3}
\end{align}
The corresponding stationary problem is a typical elliptic boundary value problem as remarked in \refeq{SA}.
 Since its resolvent problem satisfies necessary ellipticity conditions (see \cite{Ag}, \cite{Ta}, \cite{DHP03}, \cite{PS}), it is possible to solve this problem at least locally in time if one starts from some $L^p$-type ($1<p<\infty$) Besov initial data.
 However, since a homogeneous function may not belong to such a space, one cannot apply this theory to constructing self-similar solutions.
 We would rather work on $L^\infty$ type spaces so that they accommodate homogeneous functions.
 The choice of function spaces is similar to that of \cite{AK25}, \cite{AG14}.
 However, our way to handle the no-flux boundary condition is quite different from \cite{AK25}, even for the case $N=1$.

To see our idea, let us consider the one-dimensional case.
 The problem \eqref{EU1}--\eqref{EU3} can be rewritten as
\begin{align*}
    u_t + \left( \frac1{(1+u_x^2)^{1/2}} \kappa_x \right)_x
    &= 0, \quad x>0, \quad t \in (0,T), \\
    u_x(0,\cdot) &= \tan\theta, \quad t \in (0,T), \\
    \frac1{(1+u_x^2)^{1/2}} \kappa_x (0,\cdot) &= 0, \quad t \in (0,T),
\end{align*}
where $\kappa=H=u_{xx}/(1+u_x^2)^{3/2}$ and $f_x=\partial f/\partial x$.
 For a given $\gamma_0=\tan\theta_0\in\mathbb{R}$,
% 原稿 2025/12/9、1追-3/4
we consider $v=u-\gamma_0x$.
 Then $v$ must satisfy
\begin{align}
    v_t + (\beta v_{xxx} - G)_x
    &= 0, \quad x>0, \quad t \in (0,T), \label{E1D1} \\
    v_x(0,\cdot) &= \tan\theta - \tan\theta_0, \quad t \in (0,T),  \label{E1D2} \\
    (\beta v_{xxx} - G) (0,\cdot) &= 0, \quad t \in (0,T), \label{E1D3}
\end{align}
where $\beta=1/(1+\gamma_0^2)^2$ and $G$ is a remainder term which is small if $v$ is small in $C^3$ sense.
 Note that the remainder terms in the first and the third equations are the same.
 We take this advantage to handle the nonlinear term $G$ while the paper \cite{AK25} did not use this special structure so the proof is rather long.
 If $\beta>0$ is a given constant, we actually have a simple solution formula for a given $G$ for \eqref{E1D1}--\eqref{E1D3}. 
 It is of the form
\[
    v(t) = ax + e^{-\beta t (\partial_x^4)_n} (v_0-ax)
    + \int_0^t \partial_x e^{-\beta(t-s)(\partial_x^4)_d} G(s)\, ds,
\]
where $v|_{t=0}=v_0$ and $a=\tan\theta-\tan\theta_0$; $(\partial_x^4)_n$ denotes the biharmonic operator with the ``Neumann" boundary condition $v_x=v_{xxx}=0$ while $(\partial_x^4)_d$ denotes the biharmonic operator with the ``Dirichlet" condition $v=v_{xx}=0$.
 This is easy to observe.
 Indeed, we may assume $a=0$ and $v_0=0$ by subtracting
\[
    ax+e^{-\beta t(\partial_x^4)_n} (v_0-ax)
\]
from $v$.
 We notice that
\[
    w = \int_0^t e^{-\beta(t-s)(\partial_x^4)_d} G(s)\,ds
\]
solves
\begin{align*}
    w_t + \beta w_{xxxx} - G &= 0, \quad x>0, \quad t>0, \\
    w &= 0, \quad x=0, \quad t>0, \\
    w_{xx} &= 0, \quad x=0, \quad t>0, \\
    w &= 0, \quad x>0, \quad t=0.
\end{align*}
It is easy to see that $v=w_x$ solves \eqref{E1D1}, \eqref{E1D2} with $a=\tan\theta-\tan\theta_0=0$ and $v|_{t=0}=0$.
 To check \eqref{E1D3}, we observe that $w=0$ at $x=0$ implies $w_t=0$.
 Thus
\[
    \beta v_{xxx} - G = \beta w_{xxxx} - G
    = -w_t = 0 \quad\text{at}\quad x=0,
\]
which is \eqref{E1D3}.
% 原稿 2025/12/9、1追-4/4
 Using this observation, we linearize \eqref{E1D1}--\eqref{E1D3} around zero and construct a solution in H\"older space.
 The basic idea is the same as in our paper \cite{GGK25}.
 We use classical estimates in \cite{So65} rather than Sobolev-type estimates used in \cite{KL12}.
 We further give a time-decay rate of $v$ by using scaled H\"older spaces introduced in \cite{GGK25}.

Since we are working in a multi-dimensional setting, there needs to be a control of tangential derivations $\nabla' u$.
 Although technically involved, the procedure itself is standard.

% 修正指示 2026/1/25、1/3
Compared with our results and results in \cite{AK25} for $N=1$, there are several differences.
 First of all their function spaces are different although both
% 修正指示 2026/1/25、2/3 
use H\"older spaces with weights in time.
 However, the solution in \cite{AK25} is constructed in some weighted H\"older spaces where differentiability in time is not guaranteed.
 Their solution is a mild solution which is a weak solution.
 Since our solution is smooth, the uniqueness in \cite{AK25} implies that their solution is also smooth.

The boundary problem of the surface diffusion flow equation has been mainly studied in one-dimensional setting, i.e., $N=1$, except the work of M.~G\"o\ss wein \cite{Go19}.
 Under the right angle boundary conditions, stability of part of straight lines or circles (which are equilibria) is studied in H.~Garcke, K.~Ito and Y.~Kohsaka \cite{GIK05}, \cite{GIK08}, M.~Gazwani and J.~McCoy \cite{GM}, and G.~Wheeler and V.-M.~Wheeler \cite{WW}.
% 修正指示 2026/1/25、3/3
Other than self-similar solutions constructed in \cite{AK25}, several special solutions, such as traveling wave solutions are constructed.
 For example, T.~Kagaya and Y.~Kohsaka \cite{KK} constructed non-convex traveling waves with two end points on $\mathbb{R}^1\times\{0\}$ with given contact angle and no-flux condition.
 They showed that there is a chance that the curve may not be in the half plane $\mathbb{R}^1\times\mathbb{R}_+^+$ even if the curve is in the half plane near end points.
 The existence of a non-compact, complete translating solution (soliton) was proved by W.~J.~Ogden and M.~Warren \cite{OW}.
 In both studies, the problem is reduced to analysis of ordinary differential equations of the function associated with the Gauss map.

This paper is organized as follows.
 In Section~\ref{SB}, we prepare function spaces and state our main results.
 In Section~\ref{SINT}, we derive a convenient integral form of our problem with layer and volume potentials.
 In Section~\ref{SL}, we prepare estimates of layer and volume potentials.
 In Section~\ref{SN}, we give necessary estimates for nonlinear terms.
 The proofs of the main results are given in Section~\ref{SP}.
 In \ref{SA}, we discuss some ellipticity conditions to derive a priori estimates in H\"older a nd $L^p$ spaces.
 Note that Schauder estimates (estimates in H\"older spaces) for general elliptic problems were established by S.~Agmon, A.~Douglis and L.~Nirenberg \cite{ADN59}, \cite{ADN64}.
 In \ref{SR}, % Appendix B
 we give necessary calculations to derive \eqref{EU1} from \eqref{E1}.
 
%%%%%%%%%%%%%%%%%%%%%%%%%%%%%%%%%%%%%%%%%%%%%%%%%
% 原稿 2025/12/16、2-1/5
\section{Function spaces and main results} \label{SB} % Section 2

As in \cite{GGK25}, we introduce scaled H\"older norms.
 Let $f$ be a (possibly vector or tensor valued) continuous function on $\overline{\mathbb{R}_+^N}\times(a,b)$.
 For $\lambda,\mu\in(0,1)$ we define H\"older seminorms
of $f$ in $\overline{\mathbb{R}_+^N}\times(a,b)$ by
\begin{gather*}
    \lfloor f]_{C_x^\lambda\left(\overline{\mathbb{R}_+^N}\times(a,b)\right)}
    := \sup \left\{ \frac{\left|f(x,t)-f(y,t)\right|}{|x-y|^\lambda} \biggm|
    x,y \in \overline{\mathbb{R}_+^N},\ x\neq y,\ t\in(a,b)\right\}, \\
    [f]_{C_t^\mu\left(\overline{\mathbb{R}_+^N} \times (a,b)\right)}
    := \sup \left\{ \frac{\left|f(x,s)-f(x,t)\right|}{|t-s|^\mu} \biggm|
    s,t \in(a,b),\ s\neq t,\ x\in\overline{\mathbb{R}_+^N} \right\}.
\end{gather*}
We set
\begin{align*}
    [f]_{C_x^0\left(\overline{\mathbb{R}_+^N}\times(a,b)\right)}
    &= [f]_{C_t^0\left(\overline{\mathbb{R}_+^N}\times(a,b)\right)}
    := \|f\|_{L^\infty\left(\mathbb{R}_+^N\times(a,b)\right)} \\
    &= \sup \left\{ \left|f(x,t)\right| \bigm|
    x\in\mathbb{R}_+^N,\ t\in(a,b) \right\}.
\end{align*}
For $\lambda\ge0$ and $n\in\mathbb{Z}_{\ge1}$, we next define space-time H\"older seminorms by
\begin{multline*}
    [f]_{C^{\lambda,\lambda/M}\left(\overline{\mathbb{R}_+^N}\times(a,b)\right)}
    := \sum_{\substack{\ell,m\in\mathbb{Z}_{\ge0}\\
    \ell+Mm=\lfloor\lambda\rfloor}}[\nabla^\ell \partial_t^m f]_{C_x^{\lambda-\lfloor\lambda\rfloor}\left(\overline{\mathbb{R}_+^N}\times(a,b)\right)} \\
    + \sum_{\substack{\ell,m\in\mathbb{Z}_{\ge0}\\
    \lambda-M <\ell+Mm\le\lambda}}[\nabla^\ell \partial_t^m f]_{C_t^{(\lambda-\ell-Mm)/M}\left(\overline{\mathbb{R}_+^N}\times(a,b)\right)}.
\end{multline*}
Here $\nabla^\ell h$ denotes an $\ell$-tensor ($\mathbb{R}^{N^\ell}$-valued function) consisting of all $\ell$th spatial partial derivatives of $h$ and $\lfloor\lambda\rfloor$ denotes the largest integer less than $\lambda$.
 For $\lambda\ge0$ and $M\in\mathbb{Z}_{\ge0}$, we define H\"older norms by
\[
    \|f\|_{BC^{\lambda,\lambda/M}\left(\overline{\mathbb{R}_+^N}\times(a,b)\right)}
    := \sum_{\substack{\ell,m\in\mathbb{Z}_{\ge0}\\
    \ell+Mm\le\lambda}}\|\nabla^\ell \partial_t^m f\|_{L^\infty\left(\mathbb{R}_+^N\times(a,b)\right)}
    + [f]_{C^{\lambda,\lambda/M}\left(\overline{\mathbb{R}_+^N}\times(a,b)\right)}.
\]
% 原稿 2025/12/16、2-2/5
For a bounded interval $(a,b)$, it is convenient to introduce a scaled H\"older norm $\|f\|_{BC^{\lambda,\lambda/M}\left(\overline{\mathbb{R}_+^N}\times(a,b)\right)}$ by
\begin{align*}
    \|f\|'_{BC^{\lambda,\lambda/M}\left(\overline{\mathbb{R}_+^N}\times(a,b)\right)}
    := &\sum_{\substack{\ell,m\in\mathbb{Z}_{\ge0}\\
    \ell+Mm\le\lambda
    }} (b-a)^{\ell/M+m}
    \| \nabla^\ell \partial_t^m f\|_{L^\infty\left(\mathbb{R}_+^N\times(a,b)\right)} \\
    &+ (b-a)^{\lambda/M} [f]_{C^{\lambda,\lambda/M} \left(\overline{\mathbb{R}_+^N}\times(a,b)\right)}.
\end{align*}
Scaled H\"older norms have the scaling invariant property
\begin{equation} \label{SIP}
    \|\bar{f}\|'_{BC^{\lambda,\lambda/M}\left(\overline{\mathbb{R}_+^N}\times(0,1)\right)}
    = \|f\|_{BC^{\lambda,\lambda/M}\left(\overline{\mathbb{R}_+^N}\times(a,b)\right)},
\end{equation}
where
\[
    \bar{f}(y,s):=f\left((b-a)^{1/M}y,\ a+(b-a)s\right), \quad
    (y,s) \in \overline{\mathbb{R}_+^N} \times(0,1).
\]

We now introduce a few useful function spaces to handle equation \eqref{EU1} by taking $M=4$.
 For $\alpha\in\mathbb{R}$, $k\in\mathbb{Z}_{\ge0}$ and $\mu\in[0,1)$, we set
\[
    Z_\alpha^{k+\mu}
    := \left\{ v \in C \left(\overline{\mathbb{R}_+^N}\times(0,\infty)\right) \Bigm|
    \|v\|_{Z_\alpha^{k+\mu}} < \infty \right\}
\]
with
\[
    \|v\|_{Z_\alpha^{k+\mu}}
    := \sup_{t>0} t^{-\alpha/4}
    \|v\|'_{BC^{k+\mu,(k+\mu)/4}\left(\mathbb{R}_+^N\times(t/2,t)\right)}.
\]
Since we mainly study gradient of a solution $v$ to \eqref{EU1}, we introduce
\[
    \hat{Z}_\alpha^{k+\mu}
    := \left\{ v \in C \left(\overline{\mathbb{R}_+^N}\times(0,\infty)\right) \Bigm|
    \nabla v \in Z_\alpha^{k+\mu}\right\}.
\]
We shall write unweighted spaces simply by
\[
    Z^{k+\mu}
    := Z_0^{k+\mu}, \quad
    \hat{Z}^{k+\mu}
    := \hat{Z}_0^{k+\mu}.
\]
Similar spaces for $\mathbb{R}^{N^k}$-valued functions on $\mathbb{R}_+^N\times(0,\infty)$ and functions on $\partial\mathbb{R}_+^N\times(0,\infty)$ are defined and denoted by the same notations unless there arises confusion.

% 原稿 2025/12/16、2-3/5
For a closed set $K$ in $\mathbb{R}^N$, we denote by $\operatorname{Lip}(K)$ the space of all Lipschitz continuous functions in $K$.

We are now in a position to state our main results.
 We begin with global solvability of \eqref{EU1}--\eqref{EU3} with given initial data.
\begin{thm}\label{TM}
Let $\mu\in(0,1)$ and $\gamma_0\in\mathbb{R}$.
 Then there exist constants $\varepsilon^*>0$ and $\delta^*>0$ such that if $\gamma=\tan\theta\in\mathbb{R}$ and $u_0\in \operatorname{Lip}(\overline{\mathbb{R}^N_+})$ satisfies
\[
|\gamma-\gamma_0|<\varepsilon^*
\quad\text{and}\quad
\|\nabla u_0-\gamma_0 e_N\|_{L^\infty(\mathbb{R}^N_+)}<\varepsilon^*,
\]
then the problem \eqref{EU1}--\eqref{EU3} with $T=\infty$ admits a unique solution $u$ with initial data $u_0$ in the set
\begin{equation} \label{Esm}
      \hat{B}_{\delta^*} :=\left\{v\in\hat{Z}^{4+\mu} \Bigm|
      \|\nabla v-\gamma_0e_N\|_{Z^{4+\mu}}\le\delta^*\right\},
\end{equation}
where $e_N=(0,\ldots,0,1)\in\mathbb{R}^N$.
 Moreover, $u$ belongs to $C^\infty \left (\overline{\mathbb{R}^N_+}\times(0,\infty) \right)\cap C\left (\overline{\mathbb{R}^N_+}\times[0,\infty) \right)$ and satisfies
\begin{gather}
    \sup_{x\in\mathbb{R}^N_+} \left|\nabla^\ell \partial_t^m u(x,t) \right|
    \le Ct^{(1-\ell-4m)/4}, \label{Ede} \\
    \sup_{x\in\mathbb{R}^N_+} \left| (u-u_0)(x,t) \right| \le Ct^{1/4} \label{ER}
\end{gather}
for all $(\ell,m)\in\mathbb{Z}_{\ge 0}^2\setminus\left\{(0,0)\right\}$ and $t>0$ with some constant $C$ depending only on $N$, $\gamma$, $\varepsilon^*$, $\delta^*$, $\ell$, $m$.
\end{thm}
As corollaries of Theorem~\ref{TM}, we obtain the existence and stability of non-trivial self-similar solutions close to the trivial ones $\gamma_0x$.
 We call a function $u$ on $\mathbb{R}_+^N\times(0,\infty)$ is a self-similar solution to \eqref{EU1}--\eqref{EU3} if $u$ satisfies \eqref{EU1}--\eqref{EU3} with $T=\infty$ and self-similar, i.e., $u^\sigma=u$ in $\mathbb{R}_+^N\times(0,\infty)$ for all $\sigma>0$, where
\[
    u^\sigma(x,t) := \sigma^{-1/4} u(\sigma^{1/4}x, \sigma t).
\]
Let $S^{N-1}$ be the unit sphere in $\mathbb{R}^N$ centered at the origin, i.e.,
\[
    S^{N-1} = \left\{ x \in \mathbb{R}^N \bigm|
    |x| = 1 \right\}.
\]
% 原稿 2025/12/16、2-4/5
 Let $S_+^{N-1}$ denote a semi-sphere defined by $S_+^{N-1}=S^{N-1}\cap\mathbb{R}_+^N$.
 For a function $\psi$ defined on $\overline{S_+^{N-1}}$, let $\bar{\psi}$ denote its $1$-homogeneous extension to $\overline{\mathbb{R}_+^N
 }$, i.e., $\bar{\psi}(x)\equiv|x|\psi\left(x/|x|\right)$.
\begin{cor} \label{CSS}
Let $\mu\in(0,1)$ and $\gamma_0\in\mathbb{R}$.
 Then there exist constants $\varepsilon^*>0$ and $\delta^*>0$ such that if $\gamma=\tan\theta\in\mathbb{R}$ and $\psi\in\operatorname{Lip}(\overline{S_+^{N-1}})$ satisfies
\begin{equation} \label{ECS}
    |\gamma-\gamma_0| < \varepsilon^*
    \quad\text{and}\quad
    \|\nabla\bar{\psi} - \gamma_0 e_N\|_{L^\infty(\mathbb{R}_+^N)} < \varepsilon^*,
\end{equation}
then there exists a unique self-similar solution
\[
    u^{\gamma,\psi} \in C^\infty \left(\overline{\mathbb{R}_+^N}\times (0,\infty) \right)
    \cap C \left(\overline{\mathbb{R}_+^N} \times [0,\infty) \right)
\]
to \eqref{EU1}--\eqref{EU3} satisfying
\begin{equation} \label{EPRo}
    \lim_{r\to\infty} \frac{u^{\gamma,\psi}(r\theta,1)}{r} = \psi(\theta)
    \quad\text{for each}\quad
    \theta \in S_+^{N-1}
\end{equation}
in the set $\hat{B}_{\delta^*}$, where $\hat{B}_{\delta^*}$ is as in Theorem~\ref{TM}.
 Moreover $u^{\gamma,\psi}(x,0)=\bar{\psi}(x)$.
\end{cor}
\begin{remark} \label{RB}
If we take $\gamma_0=0$ and $\psi=0$, then self-similar solution $u^{\gamma,\psi}$ is bounded by \eqref{ER}.
 If $\gamma=\tan\theta$ is small but $\gamma\neq0$, then Corollary~\ref{CSS} implies that there exists a non-trivial bounded self-similar solution.
 This result has a strong contrast to the case of the whole space.
 In \cite{RW} it is shown that there exists no bounded self-similar solution in $\mathbb{R}$ other than zero;
 see also \cite[Theorem 7]{GK}.
\end{remark}
% 原稿 2025/12/16、2-5/5
\begin{proof}[Proof of Corollary~\ref{CSS} admitting Theorem~\ref{TM}]
This easily follows from the unique existence Theorem~\ref{TM} since it turns out that self-similar solutions are solutions of \eqref{EU1}--\eqref{EU3} with initial data $u_0=\bar{\psi}$, which is $1$-homogeneous.
 Indeed, let $\varepsilon^*$ and $\hat{B}_{\delta^*}$ be as in Theorem~\ref{TM}.
 Then the problem \eqref{EU1}--\eqref{EU3} with initial data $u_0=\bar{\psi}$ admits a unique solution $u^{\gamma,\psi}$ in $\hat{B}_{\delta^*}$,
 Moreover by the scale invariance of \eqref{EU1}--\eqref{EU3}, $(u^{\gamma,\psi})^\sigma$ is also a solution to the problem \eqref{EU1}--\eqref{EU3}.
 Furthermore, by the definition of the norm of $Z^{4+\mu}$, the set $\hat{B}_{\delta^*}$ is invariant under the scaling $u\mapsto u^\sigma$.
 Thus the uniqueness of a solution implies that $u^{\gamma,\psi}\equiv(u^{\gamma,\psi})^\sigma$ so $u^{\gamma,\psi}$ is self-similar.
 Assertion of \eqref{EPRo} is a consequence of relations $u^{\gamma,\psi}(r\theta,1)=ru^{\gamma,\psi}(\theta,r^{-4})$ and \eqref{ER}.
\end{proof}
\begin{cor} \label{CST}
Assume the same hypothesis as in Theorem~\ref{TM} concerning $\gamma$, $\gamma_0$, $\mu$ and $u_0$.
 Assume $\psi\in\operatorname{Lip}(\overline{S_+^{N-1}})$ satisfies \eqref{ECS}.
 Assume moreover that
\begin{equation} \label{EAsy}
    \lim_{r\to\infty} \frac{u_0(r\theta)}{r} = \psi(\theta)
    \quad\text{for}\quad
    \theta \in S_+^{N-1}.
\end{equation}
Then the solution $u\in\hat{B}_{\delta^*}$ to \eqref{EU1}--\eqref{EU3} with $u|_{t=0}=u_0$ satisfies
\[
    u^\sigma \to u^{\gamma,\psi}
    \quad\text{as}\quad
    \sigma \to \infty
\]
uniformly in any compact set of $\overline{\mathbb{R}_+^N}\times[0,\infty)$.
 The convergence is also uniform for all its derivatives in $\overline{\mathbb{R}_+^N}\times(0,\infty)$.
\end{cor}
This also follows from Theorem~\ref{TM}.
 We shall postpone its proof in Section~\ref{SP}.

%%%%%%%%%%%%%%%%%%%%%%%%%%%%%%%%%%%%%%%%%%%%%%%%%
\section{Integral equations} \label{SINT} % Section 3

% 原稿 2025/12/17、3-1/3
We shall derive an integral equation for a difference $v=u-\gamma_0e_N$.

For a real-valued function $v$ on $\mathbb{R}_+^N$, we set
\begin{equation} \label{EL1}
    Lv := \operatorname{tr} \left( P(\gamma_0 e_N) \nabla^2 v \right)
    =\operatorname{div} \left( P(\gamma_0 e_N) \nabla v \right),
\end{equation}
By definition $P(\gamma_0 e_N)$ is an $N\times N$ diagonal matrix of the form
\[
    P(\gamma_0 e_N) = \operatorname{diag} \left( 1,\ldots,1,\frac{1}{1+\gamma_0^2} \right).
\]
Thus
\[
    L = \sum_{i=1}^{N-1} D_i^2
    + \frac{1}{1+\gamma_0^2} D_N^2.
\]
To write the top order term in \eqref{EU1}, for a $\mathbb{R}^N$-valued function $w$ on $\mathbb{R}_+^N$, we introduce
\begin{equation} \label{EA1}
    \left(\mathbb{A}[w]T \right)_\ell
    :=\sum_{1\le i,j,k\le N}a_{ijk\ell}(w)T_{ijk}
    \quad\text{(}1\le\ell\le N\text{)}
\end{equation}
with
\[
    a_{ijk\ell}(w)
    := P(\gamma_0 e_N+w)_{ij} P(\gamma_0 e_N+w)_{k\ell}
    - P(\gamma_0 e_N)_{ij} P(\gamma_0 e_N)_{k\ell},
\]
for $1\le i,j,k,\ell\le N$, where $T$ is $3$-tensor in $\mathbb{R}^N$.
 We further introduce
\begin{align*}
    B[w] &:=\frac{1}{\sqrt{1+|\gamma_0 e_N+\nabla w|^2}} P(\gamma_0 e_N+w)\nabla w, \quad
    w: \mathbb{R}_+^N \to \mathbb{R}^N, \\
    C[p] &:=\frac{1}{1+\gamma_0^2}
    \left( \gamma\sqrt{1+|p'|^2} - \gamma_0 \right), \quad
    p=(p',p_N)\in \mathbb{R}^N, \\
    p' &= (p_1,\ldots, p_{N-1}) \in \mathbb{R}^{N-1}, \\
% 原稿 2025/12/17、3-2/3
    F[w] &:= \left( 2\left( B[w]\right)^2 + \operatorname{tr} \left( B[w] \right) B[w] \right) w, \quad
    w: \mathbb{R}_+^N \to \mathbb{R}^N.
\end{align*}
\begin{lem} \label{LEP}
For $v=u-\gamma_0 e_N$, \eqref{EU1} is
\begin{equation} \label{EV1}
    v_t + L^2 v + \operatorname{div} \left(\mathbb{A}[\nabla v] \nabla^3 v
    - F[\nabla v] \right) = 0.
\end{equation}
The condition \eqref{EU2} is of the form
\begin{equation} \label{EV2}
    \frac{1}{1+\gamma_0^2} D_N v
    - C[\nabla v] = 0.
\end{equation}
The condition \eqref{EU3} is of the form
\begin{equation}\label{EV3}
    \frac{1}{1+\gamma_0^2}D_N Lv
    + \left(\mathbb{A}[\nabla v] \nabla^3 v
    - F[\nabla v] \right)_N = 0.
\end{equation}
\end{lem}
\begin{proof}
It is easy to see the equivalence of \eqref{EU2} and \eqref{EV2}.
 To see the equivalence of \eqref{EU1} and \eqref{EV1} and also the equivalence of \eqref{EU3} and \eqref{EV3}, it suffices to prove that
\begin{equation} \label{EVK}
    \omega P(\nabla u) \nabla\operatorname{tr} \left(P(\nabla u)\omega^{-1} \nabla^2 u \right)
    = P(\gamma_0 e_N) \nabla Lv
    + \mathbb{A}[\nabla v] \nabla ^3 v
    - F[\nabla v],
\end{equation}
where $\omega:=\left(1+|\nabla u|^2\right)^{1/2}$($=\left(1+|\nabla v|^2\right)^{1/2}$).
 We note the mean curvature
\[
    H:= \operatorname{tr} \left( \frac{P(\nabla u)}{\omega}\nabla^2 u \right)
    = \frac{1}{\omega} \left( \Delta u - \frac{\nabla^2 u \nabla u \cdot\nabla u}{\omega^2} \right).
\]
Since $\nabla\omega^{-k}=-k\omega^{-(k+1)}\nabla^2 u\nabla u$, $k\in\mathbb{Z}_{\ge1}$, we see that
\begin{align*} % 3行目以降は1行目の＝に揃える
    &\omega\nabla H = - \frac{H\nabla^2 u \nabla u}{\omega}
    + \nabla \left(\Delta u - \frac{\nabla^2 u\nabla u\cdot\nabla u}{\omega^2
    } \right), \\
    & \nabla \frac{\nabla^2 u\nabla u\cdot \nabla u}{\omega^2}
    - \frac{1}{\omega^2} \sum_{i,j} D_i u D_j u \nabla D_{ij} u \\
    &\hspace{2.3em}= \frac{2}{\omega^2} (\nabla^2 u)^2 \nabla u
    - \frac{2(\nabla^2 u\nabla u\cdot\nabla u)\nabla^2 u\nabla u}{\omega^4} \\
    &\hspace{2.3em}= \frac{2}{\omega^2} \nabla^2 u \left(\nabla^2 u \nabla u
    - \frac{\nabla^2 u\nabla u\cdot\nabla u}{\omega^2} \nabla u \right) \\
    &\hspace{2.3em}= \frac{2}{\omega^2} \nabla^2 u \left( I - \frac{\nabla u\otimes\nabla u}{\omega^2} \right)
    \nabla^2 u \nabla u 
    = \frac{2}{\omega^2} \nabla^2 u P(\nabla u) \nabla^2 u \nabla u,
\end{align*}
where $D_{ij}=D_iD_j$.
 Thus,
% 原稿 2025/12/17、3-3/3
\[
    \omega\nabla H = \nabla \Delta u - \sum_{i,j} \frac{D_iu D_ju \nabla D_{ij}u}{\omega^2}
    - \frac{H\nabla^2 u \nabla u}{\omega}
    - 2 \frac{\nabla^2 u}{\omega^2} P(\nabla u) \nabla^2 u \nabla u.
\]
If we set a matrix $\widetilde{B}=P(\nabla u)\nabla^2 u/\omega$, then 
\[
    P(\nabla u)\omega\nabla H = P(\nabla u) \left(\nabla\Delta u - \sum_{i,j}\frac{D_iu D_ju \nabla D_{ij}u}{\omega^2} \right)
    - \left( (\operatorname{tr}\widetilde{B})\widetilde{B}+2\widetilde{B}^2 \right) \nabla u
\]
since $H=\operatorname{tr}\widetilde{B}$.
 The $\ell$th component of the first term of the right-hand side is
\begin{align*}
      &\left( P(\nabla u) \left( \nabla\Delta u
      - \sum_{i,j}\frac{D_iu D_ju \nabla D_{ij}u}{\omega^2} \right) \right)_\ell \\
      =&\sum_k P(\nabla u)_{\ell k} \left( D_k\Delta u -\sum_{i,j}\frac{D_iu D_ju D_{ijk}u}{\omega^2} \right) \\
      =&\sum_k P(\nabla u)_{\ell k} \sum_{i,j} P(\nabla u)_{ij} D_{ijk}u \\
      =&\sum_{i,j,k} \left( P(\gamma_0 e_N)_{ij}P(\gamma_0 e_N)_{\ell k}
      + a_{ijk\ell}(\nabla v) \right) D_{ijk}v \\
    =& \left( P(\gamma_0 e_N)\nabla \left( \operatorname{tr}P(\gamma_0 e_N)\nabla^2 u \right)
    + \mathbb{A}[\nabla v]\nabla^3 v \right)_\ell \\
    =& \left(P(\gamma_0 e_N) Lv
    + \mathbb{A}[\nabla v]\nabla^3 v \right)_\ell,
\end{align*}
where $D_{ijk}=D_iD_jD_k$.
 Since $\widetilde{B}=B[\nabla v]$, we see that
\[
    F[\nabla v] = \left( (\operatorname{tr}\widetilde{B})\widetilde{B} + 2\widetilde{B}^2 \right) \nabla v,
\]
we now conclude \eqref{EVK}.
\end{proof}
% 原稿 2026/1/1、3-4 1/6
We note that \eqref{EV1}, \eqref{EV2}, \eqref{EV3} can be written as
\begin{align}
    v_t + L^2 v - \operatorname{div}G = 0
    &\quad\text{in}\quad \mathbb{R}_+^N\times (0,T), \label{ELV1} \\
    \frac{1}{1+\gamma_0^2} D_N v - h = 0
    &\quad\text{on}\quad \partial\mathbb{R}_+^N\times (0,T), \label{ELV2} \\
    \frac{1}{1+\gamma_0^2} D_N Lv - G_N = 0
    &\quad\text{on}\quad \partial\mathbb{R}_+^N\times (0,T), \label{ELV3}
\end{align}
with $G=-\left(\mathbb{A}[\nabla v]\nabla^3v-F[\nabla v]\right)$, $h=C[\nabla v]$.
 Since $L=\operatorname{div}P(\gamma_0e_N)\nabla$, \eqref{EV1}, \eqref{EV3} can be written as
\begin{align*}
    v_t + \operatorname{div} \left(P(\gamma_0e_N)\nabla(Lv)-G \right) &=0 \\
    \left(P(\gamma_0e_N)\nabla(Lv)-G \right)_N &=0.
\end{align*}
We use this special structure to write a solution in its integral form.

Let $b(\cdot,t)$ be a fundamental solution of a (anisotropic) biharmonic operator of $L^2$.
 In other words,
\[
    b(x,t) := \frac{1}{(2\pi)^N} \int_{\mathbb{R}^N} e^{ix\cdot\xi} e^{-t\left(P(\gamma_0e_N)\xi\cdot\xi\right)^2} d\xi.
\]
We set
\begin{align*}
    b_n(x,y,t) &:= b(x-y,t) + b(x-y^*, t), \\
    b_d(x,y,t) &:= b(x-y,t) - b(x-y^*, t),
\end{align*}
where $y^*=y-2y_Ne_N$, a mirror image of $y$ with respect to $\partial\mathbb{R}_+^N$.
% 原稿 2026/1/1、3-5 2/6
 We set
\begin{align*}
    (e^{-tL_n^2}g)(x) &:= \int_{\mathbb{R}_+^N} b_n(x,y,t) g(y)\, dy \\
    (e^{-tL_d^2}g)(x) &:= \int_{\mathbb{R}_+^N} b_d(x,y,t) g(y)\, dy.
\end{align*}
Since
\begin{align*}
    (e^{-tL_n^2}g)(x) &= \int_{\mathbb{R}^N} b(x-y,t) g_\mathrm{even}(y)\, dy \\
    (e^{-tL_d^2}g)(x) &= \int_{\mathbb{R}^N} b(x-y,t) g_\mathrm{odd}(y)\, dy,
\end{align*}
where $g_\mathrm{even}$ (resp.\ $g_\mathrm{odd}$) is the even (resp.\ odd) extension of $g$ for $y_N<0$, i.e., $g_\mathrm{even}(y)=g(y^*)$, $g_\mathrm{odd}(y)=-g(y^*)$ for $y_N<0$, the function $v=e^{-tL_n^2}g$ satisfies the Neumann boundary condition
\[
    D_N v = 0, \quad
    D_N Lv = 0
\]
while $v=e^{-tL_d^2}g$ satisfies the Dirichlet boundary condition
\[
    v=0, \quad Lv = 0.
\]
Of course both $v$ satisfy
\[
    v_t + L^2 v = 0 \quad\text{in}\quad
    \mathbb{R}_+^N \times (0,\infty), \quad
    \left. v \right|_{t=0} = 0 \quad\text{in}\quad
    \mathbb{R}_+^N.
\]
For an $\mathbb{R}^N$-valued function $f=(f_1,\ldots,f_N)$ on $\mathbb{R}_+^N$ we set
\[
    e^{-tL_v^2}f := \left( e^{-tL_n^2}f_1, \ldots, e^{-tL_n^2}f_{N-1}, e^{-tL_d^2}f_N \right).
\]

We shall write a solution of \eqref{ELV1}--\eqref{ELV3} for given $G$ and $h$ by volume and layer potentials.
 We consider volume potentials
\[
    (V_n f)(t) := \int_0^t e^{-(t-s)L_n^2} f(s)\, ds
\]
for a scalar function $f$, and
\[
    (WG)(t) := \operatorname{div} \int_0^t e^{-(t-s)L_v^2} G(s)\, ds
\]
for a vector-valued function $G$.
 We also use a layer
% 原稿 2026/1/1、3-6 3/6
potential
\[
    (Sh)(t) := \int_0^t E(t-s)h(s)\, ds
\]
with
\[
    \left(E(t)h\right)(x) := \int_{\partial\mathbb{R}_+^N} L_x b_n (x,y,t) h(y)\, d\sigma(y),
\]
where $\sigma$ denotes the surface measure on $\partial\mathbb{R}_+^N$.
 A key observation is that a solution $v$ to the problem
\begin{align*}
    v_t + L^2 v &= \operatorname{div}G
    \quad\text{in}\quad
    \mathbb{R}^N \times (0,T), \\
    D_N v &= 0
    \quad\text{on}\quad
    \partial\mathbb{R}_+^N \times (0,T), \\
    \frac{1}{1+\gamma_0^2} D_N Lv &= G_N
    \quad\text{on}\quad
    \partial\mathbb{R}_+^N \times (0,T),
\end{align*}
with $v|_{t=0}=0$ is represented as $WG$, i.e., $v=WG$.
 This observation simplifies nonlinear estimates.

Our representation formula is summarized as follows.
 Although we give a version for $T=\infty$, it is trivial to write it in the case $T<\infty$.
 We shall write $v(t)=v(\cdot,t)$, $G(t)=G(\cdot,t)$ in the next lemma.
\begin{lem} \label{LRep}
Assume that $v\in Z^4\cap C\left(\overline{\mathbb{R}_+^N}\times[0,\infty)\right)$ satisfies \eqref{ELV1}--\eqref{ELV3} with $v|_{t=0}=v_0$.
 Then $v$ is represented as
\begin{equation} \label{ERep}
    v(t) = e^{-tL_n^2}v_0
    + (Sh)(t) - (WG)(t).
\end{equation}
\end{lem}
\begin{proof}
Since $b$ is a rapidly decreasing function at the space infinity, we may freely use integration by parts in spatial variables for integrals appearing in this proof although $\mathbb{R}_+^N$ is unbounded.
 As in the case of the heat equation, we start with
% 原稿 2026/1/1、3-7 4/6
\begin{equation} \label{EGr}
    0 = V_n (\operatorname{div}G - v_t - L^2 v).
\end{equation}
Integrating by parts yields
\begin{align*}
    \left( V_n (\operatorname{div}G) \right)(\cdot,t) =
    & -\int_0^t \int_{\partial\mathbb{R}_+^N} b_n(\cdot, y, t-s) G_N(y,s)\, d\sigma(y)\,ds \\
    & -\int_0^t \int_{\mathbb{R}_+^N} \nabla_y b_n(\cdot, y, t-s) \cdot G(y,s)\, dy\, ds.
\end{align*}
Since 
\[
  D_{y_k}b_n(x,y,t)=\begin{cases}
    D_{x_k}b_n(x,y,t), & k=1, \ldots, N-1, \\
    D_{x_N}b_d(x,y,t), & k=N,
  \end{cases}
\]
we observe that
\begin{align*}
    &\int_0^t \int_{\mathbb{R}^N_+} \nabla_y b_n (\cdot,y,t-s) \cdot G(y,s)\,dy\,ds \\
    =& \int_0^t \left( \sum_{k=0}^{N-1} D_k e^{-(t-s)L_n^2} G_k(s)
    + D_N e^{-(t-s)L_d^2}G_N(s) \right) ds \\
    =& (WG)(t).
\end{align*}
We notice that
\begin{align*}    
    (V_n v_t)(\cdot,t) &= v(t) - e^{-tL_n^2}v_0
    + \int_0^t \int_{\mathbb{R}_+^N} \partial_t b_n(\cdot,y,t-s)v(y,s)\,dy\,ds \\
    &= v(t) - e^{-tL_n^2}v_0
    - \int_0^t \int_{\mathbb{R}_+^N} L_y^2 b_n(\cdot,y,t-s)v(y,s)\,dy\,ds.
\end{align*}
The identity \eqref{EGr} now yields
\begin{equation}
\begin{split} \label{ERep1}
    0= &-\int_0^t\int_{\partial\mathbb{R}_+^N} b_n (\cdot,y,t-s) G_N(y,s)\,d\sigma(y)\,ds 
    - (WG)(t) - v(t) + e^{-tL_n^2}v_0 \\
    &+ \int_0^t \int_{\mathbb{R}_+^N} \left\{ L_x^2 b_n(\cdot,y,t-s) v(y,s)
    - b_n(\cdot,y,t-s)L^2 v(y,s) \right\} dy\,ds. 
\end{split}
\end{equation}
% 原稿 2026/1/1、3-8 5/6
Integrating by parts, we observe that
\begin{align*}
    &\int_{\mathbb{R}_+^N} L_y^2 b_n(\cdot,y,t-s)v(y,s)\,dy \\
    =& -\int_{\partial\mathbb{R}^N_+} \left( P(\gamma_0e_N) \nabla_y L_y b_n(\cdot,y,t-s) \right)_N v(y,s)\,d\sigma(y) \\
    & - \int_{\mathbb{R}^N_+} P(\gamma_0e_N) \nabla_y L_y b_n(\cdot,y,t-s) \cdot \nabla v(y,s)\,dy \\
    =& -\frac{1}{1+\gamma_0^2} \int_{\partial\mathbb{R}^N_+} D_{y_N} b_n(\cdot,y,t-s) v(y,s)\,d\sigma(y)
    - \int_{\mathbb{R}^N_+} \nabla_y L_y b_n(\cdot,y,t-s) \cdot P(\gamma_0 e_N) \nabla v(y,s)\,dy.
\end{align*}
Integrating the second term by parts, we arrive at
\begin{align*}
    &\int_{\mathbb{R}_+^N} L_y^2 b_n(\cdot,y,t-s) v(y,s)\,dy \\
    =& -\frac{1}{1+\gamma_0^2} \int_{\partial\mathbb{R}_+^N}\left\{D_{y_N} L_y b_n(\cdot,y,t-s) v(y,s)
    -L_y b_n(\cdot,y,t-s) D_N v(y,s) \right\}\,d\sigma(y)\, \\
    &+ \int_{\mathbb{R}_+^N}L_y b_n(\cdot,y,t-s) L_y v(y,s)\,dy.
\end{align*}
Since $D_{y_N}b_n(x,y,t)=D_{y_N}^3b_n(x,y,t)=0$ for $y\in\partial\mathbb{R}_+^N$, we observe that
\[
    D_{y_N}L_y b_n(\cdot,y,t-s) = 0
    \quad\text{on}\quad
    \partial\mathbb{R}_+^N.
\]
Thus
\begin{align}
        \int_{\mathbb{R}_+^N}L_y^2 b_n&(\cdot,y,t-s) v(y,s)\, dy \\
    &= \int_{\partial\mathbb{R}_+^N}L_y b_n(\cdot,y,t-s)h(y,s)\,d\sigma(y) 
    + \int_{\mathbb{R}_+^N} L_y b_n(\cdot,y,t-s)L_y v(y,s)\,dy \\
    &= (Sh)(t) + \int_{\mathbb{R}_+^N}L_y b_n(\cdot,y,t-s)L_y v(y,s)\,dy.\label{ERep2}
\end{align}
% 原稿 2026/1/1、3-9 6/6
Similarly, integrating by parts yields
 \begin{align}
    &\int_{\mathbb{R}_+^N}b_n(\cdot,y,t-s)L_y^2 v(y,s)\,dy \notag \\
    =& \frac{1}{1+\gamma_0^2}\int_{\partial\mathbb{R}_+^N}\left\{D_{y_N}b_n(\cdot,y,t-s)Lv(y,s)-b_n(\cdot,y,t-s)D_n Lv(y,s) \right\}\,d\sigma(y) \notag \\
    &+ \int_{\mathbb{R}_+^N}L_y b_n(\cdot,y,t-s)L_y v(y,s)\,dy \notag \\
    =& -\int_{\partial\mathbb{R}_+^N}b_n(\cdot,y,t-s) G_N(y,s)\,dy
    +\int_{\mathbb{R}_+^N}L_y b_n(\cdot,y,t-s)L_y v(y,s)\,dy. \label{ERep3}
  \end{align}
Combining \eqref{ERep1}--\eqref{ERep3}, we now conclude that
\begin{align*}
  0 =& -\int_0^t\int_{\partial\mathbb{R}_+^N} b_n (\cdot,y,t-s) G_N(y,s)\,d\sigma(y)\,ds - (WG)(t) \\
  &-v(t) + e^{-tL_n^2} v_0 + (Sh)(t)
  + \int_0^t \int_{\partial\mathbb{R}_+^N} b_n(\cdot,y,t-s) G_N(y,s)\,d\sigma(y)\,ds,
\end{align*}
which is exactly the formula \eqref{ERep}.
\end{proof}

%%%%%%%%%%%%%%%%%%%%%%%%%%%%%%%%%%%%%%%%%%%%%%%%%
\section{Linear estimates} \label{SL} % Section 4

% 原稿 2025/12/21、4-1/5
We shall derive a few estimates for volume and layer potentials in our rescaled H\"older norms.
\begin{lem} \label{LLV}
Let $k\in\mathbb{Z}_{\ge 0}$, $\alpha>-4$, $\mu\in(0,1)$.
 For $f\in Z_\alpha^{k+\mu}$, set
\[
    (V_n f)(t):=\int_0^t e^{-(t-s)L_n^2}f(s)\,ds
\]
Then $V_n$ is a well-defined bounded linear operator from $Z_\alpha^{k+\mu}$ to $Z_{\alpha+4}^{k+4+\mu}$.
 In other words, the estimate
\begin{equation} \label{ELV}
    \|V_n f\|_{Z_{\alpha+4}^{k+4+\mu}}
    \le C \|f\|_{Z_\alpha^{k+\mu}}
\end{equation}
holds with some constant $C$ depending only on $k$, $\alpha$, $\mu$.
 Similarly, $V_d$, $V_v$ are bounded from $Z_\alpha^{k+\mu}$ to $Z_{\alpha+4}^{k+4+\mu}$, where
 \[
    V_d f
    = \int_0^t e^{-(t-s)L_d^2} f(s)\, ds, \quad
    V_v G
    = \int_0^t e^{-(t-s)L_v^2} G(s)\, ds.
\]
\end{lem}
\begin{proof}
We fix $\widetilde{t}>0$ and let $t\in(4\widetilde{t},8\widetilde{t})$.
 Set
\[
    I:=\int_0^{\widetilde{t}}e^{-(t-s)L_n^2}f(s)\,ds,\quad I\!I:=\int_{\widetilde{t}}^te^{-(t-s)L_n^2}f(s)\,ds.
\]
We first estimate $I$. Since $\tau-s>t/4$ for $s\in(0,\widetilde{t})$ and $\tau\in(t/2,t)$, it follows from a semigroup estimate for the biharmonic heat semigroup that
\[
    \left\|\nabla^{\ell}e^{-(\tau-s)L_n^2}f(s)\right\|_{L^\infty(\mathbb{R}_+^N)}
    \le C(\tau-s)^{-\ell/4} \left\|f(s)\right\|_{L^\infty(\mathbb{R}_+^N)}
    \le Ct^{-\ell/4}s^{\alpha/4}\|f\|_{Z^{k+\mu}_{\alpha}}
\]
for $s\in(0,\widetilde{t})$, $\tau\in(t/2,t)$, and $\ell\in\mathbb{Z}_{\ge 0}$.
 (Since $e^{-tL_n^2}g=b*g_\mathrm{even}$, such estimates are reduced to the case of the whole space which are well known; see e.g.\ \cite{KL12}, \cite{DY23}.)
 This implies that $I$ is well-defined, and the estimate
\[
    \|\nabla^{\ell}I\|_{L^\infty(\mathbb{R}_+^N\times(t/2,t))}\le Ct^{(\alpha+4-\ell)/4}\|f\|_{Z^{k+\mu}_{\alpha}}
\]
holds for $\ell\in\mathbb{Z}_{\ge 0}$.
 Furthermore, since $\partial_tI=-L^2I$, we see that
\[
    \|\nabla^\ell\partial_t^mI\|_{L^\infty\left(\mathbb{R}_+^N\times(t/2,t)\right)}
    \le Ct^{(\alpha+4-\ell-4m)/4}\|f\|_{Z^{k+\mu}_{\alpha}}
\]
for $\ell,m\in\mathbb{Z}_{\ge 0}$.
 By interpolation, we obtain
\begin{equation}\label{EEsI}
      \|I\|_{BC^{k+4+\mu,(k+4+\mu)/4}\left(\mathbb{R}_+^N\times(t/2,t)\right)}'
      \le Ct^{(\alpha+4)/4}\|f\|_{Z^{k+\mu}_{\alpha}}.
\end{equation}
We next estimate $I\!I$.
 We see that $I\!I$ solves
\begin{align*}
    &\partial_tI\!I=-L^2I\!I+f
    \quad\text{in}\quad\mathbb{R}_+^N\times(\widetilde{t},t),\\
    &D_NI\!I =0,\quad
    D_NL I\!I  =0 \quad\text{on}\quad\partial\mathbb{R}_+^N\times(\widetilde{t},t),\\
    &I\!I(\cdot,\widetilde{t})=0\quad\text{in}\quad\mathbb{R}_+^N.
\end{align*}
We use the parabolic Schauder estimate.
 However, instead of estimating $I\!I$ itself, we estimate its scaled function defined by
\[
    \widetilde{I\!I}(y,s) :=I\!I\left((t-\widetilde{t})^{1/4}y,\ \widetilde{t}+(t-\widetilde{t})s \right),\ 
    (y,s) \in \overline{\mathbb{R}_+^N}\times(0,1).
\]
This function $\widetilde{I\!I}$ satisfies
\begin{align*}
    &\partial_s\widetilde{I\!I} = -L^2\widetilde{I\!I} + \widetilde{f}
    \quad\text{in}\quad
    \mathbb{R}_+^N\times(0,1) \\
    &D_N\widetilde{I\!I} = 0, \quad
    D_N \widetilde{LI\!I}=0
    \quad\text{on}\quad
    \partial\mathbb{R}_+^N\times(0,1) \\
    &\widetilde{I\!I}(\cdot,0) = 0
    \quad\text{in}\quad
    \mathbb{R}_+^N,
\end{align*}
where
\[
    \widetilde{f}(y,s) := (t-\widetilde{t})
    f\left( (t-\widetilde{t})^{1/4}y,\ 
    \widetilde{t}+(t-\widetilde{t})s \right),\quad
    (y,s) \in \overline{\mathbb{R}_+^N}\times(0,1).
\]
(See also \cite[Section~2.3]{GGK25}.)
 Using the parabolic Schauder estimate up to the boundary \cite{So65} (see also Proposition~\ref{PSchauder}), we obtain
\begin{equation*}
    \|\widetilde{I\!I}\|'_{BC^{k+4+\mu,(k+4+\mu)/4}
    \left(\overline{\mathbb{R}_+^N}\times(0,1)\right)} 
    \le C \|f\|'_{BC^{k+\mu,(k+\mu)/4}\left(\overline{\mathbb{R}_+^N}\times(0,1)\right)}.
\end{equation*}
This together with the scale invariant property \eqref{SIP} and $\widetilde{t}\in(t/8, t/4)$ implies that
\begin{equation}\label{EEsII}
    \|I\!I\|_{BC^{k+4+\mu,(k+4+\mu)/4}\left(\mathbb{R}_+^N\times(\widetilde{t},t)\right)}'\le Ct\|f\|_{BC^{k+\mu,(k+\mu)/4}\left(\mathbb{R}_+^N\times(\widetilde{t},t)\right)}'\le Ct^{(\alpha+4)/4}\|f\|_{Z^{k+\mu}_{\alpha}}.
\end{equation}
Combining \eqref{EEsI} and \eqref{EEsII}, we obtain \eqref{ELV}.
 This completes the proof of Lemma~\ref{LLV}.
\end{proof}
%
% 原稿 2025/12/21、4-2/5
\begin{lem}[Estimate for a layer potential] \label{LLS}
Let $k\in\mathbb{Z}_{\ge 0}$, $\alpha>-4$, $\mu\in(0,1)$.
 For $h\in Z^{k+\mu}_\alpha$ defined on $\partial\mathbb{R}_+^N\times(0,\infty)$, define a function $Sh$ on $\mathbb{R}_+^N\times(0,\infty)$ by
\[
    (Sh)(t):=\int_0^t E(t-s)h(s)\,ds,\quad t>0
\]
where
\[
    \left( E(t)h \right)(x)
    := \int_{\partial\mathbb{R}^N_+} L_x b_n(x,y,t) h(y) \,d\sigma(y), \quad
    x \in \mathbb{R}^N_+.
\]
Then $S$ is a well-defined bounded linear operator from $Z^{k+\mu}_\alpha$ to $Z^{k+1+\mu}_{\alpha+1}$.
 In other words,
\begin{equation} \label{ELS}
    \|Sh\|_{ Z^{k+1+\mu}_{\alpha+1}}
    \le C\|h\|_{Z^{k+\mu}_\alpha}
\end{equation}
with some constant $C$ depending only on $k$, $\alpha$, $\mu$.
\end{lem}
\begin{proof}
Let $\overline{h}(x,t):=h\left((x',0),t\right)$.
 It clearly holds that $\|\overline{h}\|_{Z^{k+\mu}_\alpha}\le\|h\|_{Z^{k+\mu}_\alpha}$.
 It follows from the definition of $E(t)$ that
\[
      E(t-s)h(s)=-\int_{\mathbb{R}^N_+}D_{y_N}L_x b_n (x,y,t-s)\overline{h}(y,s)\,dy=
      \int_{\mathbb{R}^N_+}D_{x_N}L_x b_d (x,y,t-s)\overline{h}(y,s)\,dy.
\]
This together with the semigroup estimate that
\[
    \left\|E(t-s)h(s)\right\|_{L^\infty(\mathbb{R}^N_+)}
    \le C(t-s)^{-3/4}\left\|\overline{h}(s)\right\|_{L^\infty(\mathbb{R}^N_+)}
    \le C(t-s)^{-3/4}s^{\alpha}\|h\|_{Z^{k+\mu}_\alpha}.
\]
Thus $Sh$ is well-defined and
\begin{equation}\label{EYP}
    Sh(t)=-D_NL\int_0^t e^{-(t-s) L_d^2}\overline{h}(s)\,ds.
\end{equation}
Furthermore, the estimate \eqref{ELV} implies \eqref{ELS}.
\end{proof}
%
% 原稿 2025/12/21、5-1/5
%%%%%%%%%%%%%%%%%%%%%%%%%%%%%%%%%%%%%%%%%%%%%%%%%
\section{Estimates of nonlinear terms} \label{SN} % Section 5

In Section~\ref{SINT}, we reduced the problem \eqref{EU1}--\eqref{EU3} with initial data $v_0$ to an integral equation
\begin{equation}\label{EInt1}
\begin{aligned}
    v(t)=e^{-tL_n^2}v_0
    &+\int_0^t E(t-s)C\left[\nabla v(s)\right]\,ds\\
    &-\operatorname{div}\int_0^t e^{-(t-s)L_v^2}\left(\mathbb{A}[\nabla v]\nabla^3v-F[\nabla v]\right)(s)\,ds.
\end{aligned}
\end{equation}
We do not solve this equation directly.
 Instead we shall solve the equation for its gradient $w=\nabla v$ which is obtained by differentiating \eqref{EInt1}.
 Its explicit form is
\begin{equation*}
\begin{aligned}
    w(t)=e^{-tL_v^2}\nabla v_0
    &+ \nabla \int_0^t E(t-s)C\left[w(s)\right]\,ds\\
    &- \nabla\operatorname{div}\int_0^t e^{-(t-s)L_v^2}\left(\mathbb{A}[w]\nabla^2 w - F[w]\right)(s)\,ds
\end{aligned}
\end{equation*}
or
\begin{equation} \label{EInt2}
    w(t) = e^{-tL_v^2}\nabla v_0
    + \nabla S\left(C[w]\right)
    -\nabla\operatorname{div}V_v \left(\mathbb{A}[w]\nabla^2 w - F[w]\right).
\end{equation}
If we set
\begin{equation} \label{ELam}
    \Lambda[w](t)
    :=\nabla S\left(C[w]\right)
    -\nabla\operatorname{div} V_v \left(\mathbb{A}[w]\nabla^2 w - F[w] \right),
\end{equation}
then equation \eqref{EInt2} is of the form
\[
    w(t) = e^{-tL_v^2} \nabla v_0
    + \Lambda [w].
\]
In this section we estimate $\Lambda[w]$ which is a contribution from nonlinear terms.
 Our goal in this section
% 原稿 2025/12/21、5-2/5
is to prove that $\Lambda$ is strict contraction in a small ball of $Z^{4+\mu}$.
 In other words the term is regarded as a perturbation term.
\begin{thm} \label{TPer}
Assume that $\mu\in(0,1)$.
 Then the operator $\Lambda$ in \eqref{ELam} is a well-defined nonlinear operator from $Z^{4+\mu}$ to $Z^{4+\mu}$.
 Moreover, there exists a constant $\delta_1>0$ such that
\begin{equation} \label{EConL}
    \left\| \Lambda[w_1] - \Lambda[w_2] \right\|_{Z^{4+\mu}}
    \le \frac12 \|w_1-w_2\|_{Z^{4+\mu}}
\end{equation}
for all $w_1,w_2\in B_{\delta_1}^{4+\mu}$, where
\[
    B_\delta^{k+\mu}
    := \left\{ w \in Z^{k+\mu} \Bigm|
    \|w\|_{Z^{k+\mu}}\le\delta \right\}.
\]
\end{thm}

We recall a couple of basic estimates for products and composite functions.
\begin{lem} \label{PPr}
Let $k\in\mathbb{Z}_{\ge 0}$, $\mu\in(0,1)$, and $\alpha_1,\ldots,\alpha_m\in\mathbb{R}$.
 For any $g_i\in Z^{k+\mu}_{\alpha_i}$ ($i=1,\ldots,m$), the product $g_1\ldots g_m$ belongs to $Z^{k+\mu}_{\alpha_1+\ldots+\alpha_m}$.
 Furthermore, the estimates
\begin{gather}
      \|g_1\ldots g_m\|_{Z^{k+\mu}_{\alpha_1+\ldots+\alpha_m}}\le C\|g_1\|_{Z^{k+\mu}_{\alpha_1}}\ldots\|g_m\|_{Z^{k+\mu}_{\alpha_m}}, \notag \\
\begin{align}
      \|g_1\ldots g_m &- h_1\ldots h_m\|_{Z^{k+\mu}_{\alpha_1+\ldots+\alpha_m}} \notag \\
      &\le C\sum_{i=1}^m\|g_1\|_{Z^{k+\mu}_{\alpha_1}}\ldots\|g_{i-1}\|_{Z^{k+\mu}_{\alpha_{i-1}}}\|g_i-h_i\|_{Z^{k+\mu}_{\alpha_i}}\|h_{i+1}\|_{Z^{k+\mu}_{\alpha_{i+1}}}\ldots\|h_m\|_{Z^{k+\mu}_{\alpha_m}}\label{EPr}
\end{align}
\end{gather}
hold for any $g_i,h_i\in Z^{k+\mu}_{\alpha_i}$ ($i=1,\ldots,m$).
\end{lem}
\begin{prop} \label{PCom}
Let $k\in\mathbb{Z}_{\ge 0}$, $\mu\in(0,1)$, and $A\in BC^{k+1}(\mathbb{R}^m)$.
 For any $g_1,\ldots g_m\in Z^{k+\mu}$, the composition $h(g_1,\ldots,g_m)$ belongs to $Z^{k+\mu}$.
 Furthermore, for any $R>0$, the estimates
    \begin{equation}\label{ECom}
      \begin{gathered}
      \|A(g_1,\ldots,g_m)\|_{Z^{k+\mu}}\le C,\\
      \|A(g_1,\ldots,g_m)-A(h_1,\ldots,h_m)\|_{Z^{k+\mu}}\le C\max_{i=1,\ldots,m}\|g_i-h_i\|_{Z^{k+\mu}}
      \end{gathered}
    \end{equation}
      for all $g_i,h_i\in B^{k+\mu}_R$, $i\in\{1,\ldots,m\}$.
\end{prop}
These two propositions can be proved as in \cite[Proposition 2.1]{GGK25} so we omit the proofs.

% 原稿 2025/12/21、5-3/5
As applications of Propositions~\ref{PPr} and \ref{PCom}, we obtain estimates for nonlinear terms.
\begin{lem} \label{LPer}
Let $k\in\mathbb{Z}_{\ge 0}$, $\mu\in(0,1)$, and $\delta\in(0,1)$.
 Define $B^{k+\mu}_\delta:=\{w\in Z^{k+\mu}: \|w\|_{Z^{k+\mu}}\le \delta\}$.
 Then
\begin{gather}
    \left\|\mathbb{A}[w]\right\|_{Z^{k+\mu}}\le C\delta^2,\quad
    \left\|\mathbb{A}[w_1]-\mathbb{A}[w_2]\right\|_{Z^{k+\mu}}\le C\delta\|w_1-w_2\|_{Z^{k+\mu}},\label{EAE}\\
    \left\|F[w]\right\|_{Z^{k-1+\mu}_{-2}}\le C\delta^2,\quad
    \left\|F[w_1]-F[w_2]\right\|_{Z^{k-1+\mu}_{-2}}\le C\delta\|w_1-w_2\|_{Z^{k+\mu}},\label{EFE}\\
    \left\|C[w]\right\|_{Z^{k+\mu}}\le|\gamma-\gamma_0|+C\delta^2,\quad 
    \left\|C[w_1]-C[w_2]\right\|_{Z^{k+\mu}}\le C\delta\|w_1-w_2\|_{Z^{k+\mu}},\label{ECE}
\end{gather}
for all $w,w_1,w_2\in B^{k+\mu}_\delta$.
\end{lem}
The first estimates in \eqref{EAE}--\eqref{ECE} follow from the second estimates since $\mathbb{A}[0]=0$, $F[0]=0$ and $C[0]=(1+\gamma_0^2)^{-1}(\gamma-\gamma_0)$, respectively.
\begin{proof}[Proof of Theorem~\ref{TPer}]
Let $\delta_1\in(0,1)$ and $w_1,w_2\in B^{4+\mu}_{\delta_1}$.
 It follows from \eqref{EAE}--\eqref{ECE} and \eqref{EPr} that
    \begin{equation}\label{ECon1}
    \begin{aligned}
      &\left\|\mathbb{A}[w_1]\nabla^2w_1-\mathbb{A}[w_2]\nabla^2w_2 \right\|_{Z^{2+\mu}_{-2}}\\
      &\le C\left\|\mathbb{A}[w_1]\right\|_{Z^{2+\mu}} \left\|\nabla^2(w_1-w_2)\right\|_{Z^{2+\mu}_{-2}}
      +C \left\|\mathbb{A}[w_1]-\mathbb{A}[w_2]\right\|_{Z^{2+\mu}}\|\nabla^2w_2\|_{Z^{2+\mu}_{-2}}\\
      &\le C\delta_1\left\|\nabla^2(w_1-w_2)\right\|_{Z^{2+\mu}_{-2}}+ C\delta_1\|w_1-w_2\|_{Z^{2+\mu}}\le C\delta_1\|w_1-w_2\|_{Z^{4+\mu}},
    \end{aligned}
    \end{equation}
    that
    \begin{equation}\label{ECon2}
      \left\|F[w_1]-F[w_2]\right\|_{Z^{2+\mu}_{-2}}\le C\delta_1\|w_1-w_2\|_{Z^{3+\mu}},
    \end{equation}
    and that
    \begin{equation}\label{ECon3}
      \left\|C[w_1]-C[w_2]\right\|_{Z^{4+\mu}}\le C\delta_1\|w_1-w_2\|_{Z^{4+\mu}}.
    \end{equation}
    Combining \eqref{ECon1}--\eqref{ECon3} with Lemmas~\ref{LLV} and \ref{LLS}, we obtain
    \[
    \left\|\Lambda[w_1]-\Lambda[w_2]\right\|_{Z^{4+\mu}}\le C\delta_1\|w_1-w_2\|_{Z^{4+\mu}}.
    \]
    Taking $\delta_1>0$ sufficiently small, we deduce \eqref{EConL} for $w_1,w_2\in B^{4+\mu}_{\delta_1}$.
\end{proof}

%%%%%%%%%%%%%%%%%%%%%%%%%%%%%%%%%%%%%%%%%%%%%%%%%
\section{Proof of main results} \label{SP} % Section 6

% 原稿 2025/12/23、6-1/3
In this section, we solve the integral equation \eqref{EInt1} by solving \eqref{EInt2}, and prove our main results.

We begin with discussing the relation of \eqref{EInt1} and \eqref{EInt2}.
\begin{prop}\label{PEq}
If $v\in\hat{Z}^{4+\mu}$ is a solution to the equation \eqref{EInt1}, then $w=\nabla v$ is a solution to the equation \eqref{EInt2}.
 Conversely, if $\mathbb{R}^N$-valued function $w\in Z^{4+\mu}$ is a solution to the equation \eqref{EInt2}, then
\begin{equation}\label{EvF}
    v=e^{-tL_n^2}v_0+\int_0^t E(t-s)C\left[w(s)\right]\,ds
    -\operatorname{div}\int_0^t e^{-(t-s)L_v^2} \left(\mathbb{A}[w]\nabla^2w-F[w] \right)(s)\,ds
\end{equation}
is a solution to the equation \eqref{EInt1} and satisfies $\nabla v=w$.
 In particular, $v$ belongs to $\hat{Z}^{4+\mu}$.
 Furthermore, if $w\in Z^{k+\mu}$ with $k\ge 4$, then
\begin{equation}\label{EDL}
    \|v-e^{-tL_n^2}v_0\|_{Z^{k+1+\mu}_1}
    \le C\|w\|_{Z^{k+\mu}}.
\end{equation}
\end{prop}
\begin{proof}
It is clear that if $v\in\hat{Z}^{4+\mu}$ is a solution to the equation \eqref{EInt1}, then $w=\nabla v$ is a solution to equation \eqref{EInt2}.
 Conversely, let $w\in Z^{k+\mu}$ with $k\ge 4$ be a solution to the equation \eqref{EInt2}, then it follows from Lemmas~\ref{PPr} and \ref{LPer} that
\[
    \mathbb{A}[w]\nabla^2w\in Z^{k-2+\mu}_{-2},\quad 
    F[w]\in Z^{k-1+\mu}_{-2},\quad 
    C[w]\in Z^{k+\mu}.
\]
This together with Lemmas~\ref{LLV}, \ref{LLS} and \ref{LPer} implies that the function $v$ in \eqref{EvF} is well-defined and satisfies \eqref{EDL}.
 Since $w$ satisfies \eqref{EInt2}, taking gradient of \eqref{EvF} yields $w=\nabla v$.
 Substituting $w=\nabla v$ in \eqref{EvF}, we conclude that $v$ is a solution to the equation \eqref{EInt1}.
\end{proof}
%
% 原稿 2025/12/23、6-2/3
We are now in a position to prove the main theorem.
\begin{proof}[Proof of Theorem~\ref{TM}]
By Proposition~\ref{PEq}, it suffices to find a unique solution to equation~\eqref{EInt2} in $B^{4+\mu}_\delta$ for some $\delta>0$ to prove the unique existence in $\hat{B}$.
    
Let $\delta\in(0,\delta_1)$ be a small constant to be determined later.
 By assumptions and semigroup estimates, we observe that $\|\nabla v_0\|_{Z^{4+\mu}}\le C\varepsilon^*$.
 Furthermore, it follows from Lemmas~\ref{LLV}, \ref{LLS}, \ref{LPer} that
\[
    \left \|e^{-(t-s)L_v^2}\nabla v_0+\Lambda[w] \right\|_{Z^{4+\mu}}\le C\varepsilon^*+|\gamma-\gamma_0|+C \delta^2
\]
for all $w\in B^{4+\mu}_\delta$.
 Hence, by \eqref{EConL} and the contraction mapping theorem, we find a unique solution $w\in B^{4+\mu}_\delta$ to the equation~\eqref{EInt2}, provided that $\delta_1$ is sufficiently small and $C\varepsilon_*+|\gamma-\gamma_0|<\delta/2$.

We next prove that $u\in \hat{Z}^{k+\mu}$ for any $k\in\mathbb{Z}_{\ge 4}$ by induction.
 The case $k=4$ has already been proved.
 To get higher regularity, we use Nirenberg's trick.
 Namely, we first prove regularity for tangential derivatives $\nabla'v$ of $v$ and then we use the original PDE to get regularity of normal derivatives $D_Nv$.
 Let $k_0\in\mathbb{Z}_{\ge 4}$ and assume that $u\in\hat{Z}^{k_0+\mu}$.
 It suffices to prove that $w\in Z^{k_0+1+\mu}$.
 Define the differential operator $\bar{L}$ by
\[
    \bar{L}v:=\operatorname{tr}\left(P(\nabla u)D^2v\right).
\]
Let $v:=u-\gamma x_N$ and $w:=\nabla v$.
 Noting that $\bar{L}^2w_i=L^2w_i+\operatorname{div}\left(\mathbb{A}[w]\nabla^3w_i\right)$, we deduce by differentiating \eqref{EV1} that
\begin{equation}\label{ET1}
    w_{it}+\bar{L}^2w_i=-\operatorname{div} \left(\mathbb{A}_i[w]\nabla^2w-D_iF[w]\right)
    \quad\text{in}\quad\mathbb{R}^N_+\times(0,\infty)
\end{equation}
for $i\in\{1,\ldots,N\}$, where
\[
    \left(\mathbb{A}_m[w]T\right)_{\ell}
    :=\sum_{ijk}D_ma_{ijk\ell}(w)T_{ijk}.
\]
Furthermore, noting that
\[
    P(\nabla u)\nabla \bar{L}w_i
    = P(\gamma_0 e_N) \nabla Lw_i
    +\mathbb{A}[w]\nabla^3w_i, 
\]
by \eqref{EV2} and \eqref{EV3} we calculate
\begin{align}
\begin{split}
    \left(P(\nabla u)Dw_i\right)_N
    &=(P(\nabla u)\nabla D_iu)_N=\sqrt{1+|\nabla u|^2}D_i\left(\frac{D_Nu}{\sqrt{1+|\nabla u|^2}}\right)=0,
\end{split}\label{ET2}\\
\begin{split}
    \left(P(\nabla u)\nabla\bar{L}w_i\right)_N
    &=\left(\frac{1}{1+\gamma_0^2}D_NLw_i+\mathbb{A}[w]\nabla^3w_i\right)_N\\
    &=\left(-\mathbb{A}_i[w]\nabla^2w+D_iF[w]\right)_N,\quad\text{on}\quad\partial\mathbb{R}^N_+\times(0,\infty),
\end{split}\label{ET3}
\end{align}
for $i\in\{1,\ldots,N-1\}$.
 Since
\begin{equation}\label{EHE}
    \nabla^2w\in Z^{k_0-2+\mu}_{-2},\quad P(\nabla u)\in Z^{k_0+\mu},\quad D_ma_{ijk\ell}(w)\in Z^{k_0-1+\mu}_{-1},\quad F[w]\in Z^{k_0-1+\mu}_{-2},
\end{equation}
by \eqref{EAE}, \eqref{EFE}, and \eqref{ECom}, it follows that
\[
    \mathbb{A}_i[w]\nabla^2w+D_iF[w]\in Z^{k_0-2+\mu}_{-3}.
\]
Using the Schauder estimate on \eqref{ET1}--\eqref{ET3} with scaling, as in the proof of \eqref{EEsII} (see also Remark~\ref{PPSchander}), yields $w_i\in Z^{k_0+1+\mu}$ for $i\in\{1,\ldots,N-1\}$.

It remains to prove that $w_N\in Z^{k_0+1+\mu}$.
 We deduce from \eqref{EV2} and \eqref{EV3} that
\begin{align}
       w_N&=(1+\gamma_0^2)C[w],\label{EN2}\\
       \bar{L}w_N
       &=(1+\gamma_0^2) \left( \mathbb{A}[w]\nabla^2 w + F[w] \right)_N,
       \quad\text{on}\quad\partial\mathbb{R}^N_+\times(0,\infty).\label{EN3}
\end{align}
It follows from $w_i\in Z^{k_0+1+\mu}$ for $i\in\{1,\ldots,N-1\}$ that $C[w]\in Z^{k_0+1+\mu}$.
 Combining this with \eqref{EHE} and using the Schauder estimate with scaling yields $w_N\in Z^{k_0+1+\mu}$.
 This completes the proof of $u\in\hat{Z}^{k+\mu}$ for any $k\in\mathbb{Z}_{\ge 4}$.

We next prove the estimate \eqref{Ede}.
 By \eqref{EDL}, we see that $u-\gamma_0x_N- e^{-tL_n^2}v_0\in Z^{k+1+\mu}_1$ for any $k\in\mathbb{Z}_{\ge 4}$.
 In particular, $u$ satisfies
\[
   \left| \nabla^k\partial_t^\ell(u-\gamma_0x_N- e^{-tL_n^2}v_0) \right|
   \le Ct^{(1-k-4\ell)/4}
\]
for all $(k,\ell)\in\mathbb{Z}_{\ge 0}^2$.
 Furthermore, it follows from the semigroup estimate for the biharmonic heat equation that
\[
   \left| \nabla^k\partial_t^\ell(\gamma_0x_N+ e^{-tL_n^2}v_0) \right|
   \le Ct^{(1-k-4\ell)/4}
\]
for all $(k,\ell)\in\mathbb{Z}_{\ge 0}^2\setminus\left\{(0,0)\right\}$.
 Thus \eqref{Ede} follows.
 The estimate \eqref{ER} follows by integrating \eqref{Ede} with $(\ell,m)=(0,1)$.
 The proof of Theorem~\ref{TM} is now complete.
\end{proof}
% 追加原稿 2026/1/23、1/5
\begin{proof}[Proof of Corollary \ref{CST}]
Let $u$ be the solution of \eqref{EU1}--\eqref{EU3} with $u|_{t=0}=u_0$ and $u\in B_\delta^*$.
 By definition of $u^\sigma$, we observe that
\[
    \nabla u^\sigma(x,t)
    = (\nabla u^\sigma) (\sigma^{1/4}x,\sigma t).
\]
The estimates \eqref{Ede} now yield
\[
    \left| \nabla^k \partial_t^\ell u^\sigma(x,t) \right|
    \le Ct^{(1-\ell-4k)/4}
\]
for $\sigma>0$ and $k,\ell\in\mathbb{Z}_{\ge0}^2\setminus\left\{(0,0)\right\}$.
 Moreover, for $u^\sigma$ we observe that
\begin{align*}
    \left| u^\sigma(x,s) - u^\sigma(y,t) \right|
    &\le \| \nabla u^\sigma \|_{L^\infty} |x-y|
    +\int_s^t  \left|u_t^\sigma (x,\tau) \right|\, d\tau \\
    &\le C \left(|x-y|+(t-s)^{1/4}\right), \\
    \left| u^\sigma(x,t) \right|
    &\le \left| u^\sigma(0,0) \right| + C\left( |x|+t^{1/4} \right) \\
    &\le C\left( 1+|x|+t^{1/4} \right),
\end{align*}
for $x,y\in\mathbb{R}_+^N$, $0\le s<t$, $\sigma>1$ with $C$ independent of $x$, $y$, $t$, $s$, $\sigma$.

By the Ascoli--Arzel\'a theorem with a diagonal argument, for any sequence $\sigma_k\to0$, there exists a subsequence $\{\sigma_{k_j}\}_{j=1}^\infty$ such that $u^{\sigma_{k_j}}$ converges to some $U$ (as $j\to\infty$) in the topology of $C_\mathrm{loc}^\infty\left(\overline{\mathbb{R}_+^N}\times(0,\infty)\right)\cap C_\mathrm{loc}\left(\overline{\mathbb{R}_+^N}\times[0,\infty)\right)$; 
in other words $u^{\sigma_{k_j}}\to U$ uniformly in any compact set of $\overline{\mathbb{R}_+^N}\times[0,\infty)$, and $\nabla^\ell\partial_t^mu^{\sigma_{n_j}}\to\nabla^\ell\partial_t^mU$ for any conpact set of $\overline{\mathbb{R}_+^N}\times|0,\infty)$ for all $(\ell,m)\in\mathbb{Z}_{\ge0}^2$.
 Thus $U$ must solve \eqref{EU1}--\eqref{EU3} and
\[
    U \in C^\infty\left(\overline{\mathbb{R}_+^N}\times(0,\infty)\right)
    \cap C\left(\overline{\mathbb{R}_+^N}\times[0,\infty)\right).
\]
We also observe that $U\in B_\delta^*$ since $u^\sigma\in B_\delta^*$.
% 追加2026/1/23、2/5
 By the assumption \eqref{EAsy}, we see that
\[
    U(\cdot,0) = \bar{\psi}
    \quad\text{in}\quad \mathbb{R}_+^N.
\]
Since such a solution is unique, we deduce that $U=u^{\gamma,\psi}$.
 Since the limit is independent of a choice of a subsequence, the convergence $u^\sigma\to U$ is now a full convergence.
 In other words,
\[
    u^\sigma \to u^{\gamma,\psi}
    \quad\text{as}\quad \sigma \to 0
\]
in the topology of
\[
    C_\mathrm{loc}^\infty \left(\overline{\mathbb{R}_+^N}\times(0,\infty)\right)
    \cap C_\mathrm{loc}\left(\overline{\mathbb{R}_+^N}\times[0,\infty)\right).
\]
The proof is now complete.
\end{proof}
%
%%%%%%%%%%%%%%%%%%%%%%%%%%%%%%%%%%%%%%%%%%%%%%%%%
\appendix
\def\thesection{Appendix~\Alph{section}}
\section{Remarks on the structure of equations} \label{SA} % Appendix A
\def\thesection{\Alph{section}}

% 原稿 2026/1/2、7-1/3
In this appendix, we discuss a priori estimates for linearized parabolic problems
\begin{equation}
\begin{split}\label{EApA1}
        u_t + L^2u &= f \quad\text{in}\quad
    \mathbb{R}_+^N\times(0,T), \\
    u(\cdot,0) &= a \quad\text{in}\quad
    \mathbb{R}_+^N,
\end{split}
\end{equation}
\begin{equation}
\begin{split} \label{EApA2}
    u &= g \quad\text{on}\quad
    \partial\mathbb{R}_+^N\times(0,T), \\
    Lu &= h \quad\text{on}\quad
    \partial\mathbb{R}_+^N\times(0,T),
\end{split}
\end{equation}
and
\begin{equation*}
\begin{split}
    u_t + L^2u &= f \quad\text{in}\quad
    \mathbb{R}_+^N\times(0,T), \\
    u(\cdot,0) &= a \quad\text{in}\quad
    \mathbb{R}_+^N,
\end{split}
\end{equation*}
\begin{equation}
\begin{split} \label{EApA3}
    D_N u &= g \quad\text{on}\quad
    \partial\mathbb{R}_+^N\times(0,T), \\
    D_N Lu &= h \quad\text{on}\quad
    \partial\mathbb{R}_+^N\times(0,T),   
\end{split}
\end{equation}
\eqref{ELV1}, \eqref{ELV2}, \eqref{ELV3} from a viewpoint of a general theory for elliptic and parabolic boundary value problems.

For a general elliptic boundary value problem with smooth coefficients, a general theory of a priori estimates in H\"older and Lebesgue spaces was established by S.~Agmon, A.~Douglis and L.~Nirenberg \cite{ADN59}, \cite{ADN64}.
 These two papers are the basis of modern linear elliptic theory including various preceding results.
 In \cite{ADN59} a general single equation was studied while in \cite{ADN64} a general system of equations was studied.
 We note that the book \cite{Ta} contains a self-contained explanation of the $L^p$-estimates of \cite{ADN59}.

Since a problem with constant coefficients in the half space $\mathbb{R}_+^N$ is very fundamental, we just sketch their result for this special case.
 We only consider an even order operator of the form
\[
    A(D) = \sum_{|\alpha|\le2m} a_\alpha D^\alpha, \quad
    \alpha = (\alpha_1, \ldots \alpha_N), \quad
    D^\alpha = D_1^\alpha \cdots D_N^{\alpha_N},
\]
where $|\alpha|=\alpha_1+\cdot+\alpha_N$, $\alpha_i\in\mathbb{Z}_{\ge0}$, $a_\alpha\in\mathbb{C}$, $m\in\mathbb{Z}_{\ge1}$ and $a_\alpha\in\mathbb{C}$.
 We say $A(D)$ is \emph{elliptic} if $A^0(\xi)\neq0$ for $\xi\in\mathbb{R}^N\setminus\{0\}$ where $A^0$ is the principal part of $A$, i.e.,
\[
    A^0(D) = \sum_{|\alpha|=2m} a_\alpha D^\alpha.
\]
We need a root condition.
\begin{enumerate}
 \item[(R)] For every pair of linearly independent real vectors $\xi$, $\eta$, the polynomial $A^0(\xi+\tau\eta)$ of the variable $\tau$ has equal numbers of roots with positive imaginary part and with negative imaginary part.
 (This condition is automatically fulfilled if $N\ge3$ since if $\tau$ is a root for $\xi$, $\eta$, then $-\tau$ is a root for $\xi$, $-\eta$ and the sphere in $\mathbb{R}^{N-1}$ is connected.
 In \cite[P.~130]{Ta}, instead of $\xi$, $-\eta$, it is written as $-\xi$, $-\eta$ but it seems that this is a typo.)
% 原稿 2026/1/2、7-2/3
 With an operator
\[
    B_j(D) = \sum_{|\beta|\le m_j} b_{j\beta} D^\beta, \quad
    j = 1,\ldots, m, \quad
    b_{j\beta} \in \mathbb{C},
\]
we consider the boundary value problem
\begin{align}
    A(D)u &= f \quad \text{in} \quad \Omega, \label{EEB1} \\
    B_j(D)u &= g_j \quad \text{on} \quad \partial\Omega,\ j=1,\ldots,m, \label{EEB2}
\end{align}
for $\Omega=\mathbb{R}_+^N$.
 An algebraic condition for the solvability of this problem was given by Ya.~B.~Lopatinski\u{\i} \cite{Lop} and Z.~Ya.~\v{S}apiro \cite{Sa}.
 This condition is now called a complementing condition or Lopatinski\u{\i}--\v{S}apiro condition.
 \item[(LS)] Assume (R).
 For each tangent vector $\xi$ to $\partial\Omega$, let $\tau_1(\xi),\ldots\tau_m(\xi)$ be the roots of polynomial $A^0(\xi+\tau\nu)$ with positive imaginary part, where $\nu$ is a normal vector to $\partial\Omega$.
 Then a linear combination of $\left\{B_j^0(\xi+\tau\nu)\right\}_{j=1}^m$ is divisible by $\prod_{j=1}^m\left(\tau-\tau_j(\xi)\right)$ if and only if all the coefficients vanish.
 Here $B_j^0(\xi)=\sum_{|\beta|=m_j}b_{j\beta}D^\beta$, the principal part of $B_j$.
\end{enumerate}
Let us state a core result of \cite{ADN59}.
\begin{thm} \label{TADN}
    Assume that $A$ is elliptic and satisfying (R) and (LS).
\begin{enumerate}
\item[(i)] (Schauder estimate: a priori estimates in H\"older space).
 There exists a constant $C$ depending only on $\mu\in(0,1)$, $A$, $B$ and $N$ such that
 \[
    \|u\|_{BC^{2m+\mu}(\bar{\Omega})}
    \le C \left( \|f\|_{BC^\mu(\bar{\Omega})}
    + \sum_{j=1}^m \|g_j\|_{BC^{2m-m_j}(\bar{\Omega})} \right)
 \]
 for all $u\in BC^{2m+\mu}(\bar{\Omega})$ satisfying \eqref{EEB1}, \eqref{EEB2}.
\item[(i\hspace{-0.1em}i)] ($L^p$ estimate: a priori estimate in $L^p$ space).
 There exists a constant $C$ depending only on $p\in(1,\infty)$, $A$, $B$ and $N$ such that
 \[
    \|u\|_{W^{2m,p}(\Omega)}
    \le C \left( \|f\|_{L^p(\Omega)}
    + \sum_{j=1}^m \|g_j\|_{W^{2m-m_j,p}(\Omega)} \right)
 \]
% 原稿 2026/1/2、7-3/3
 for all $u\in W^{2m,p}(\Omega)$ satisftying \eqref{EEB1}, \eqref{EEB2}.
\end{enumerate}
\end{thm}

As \cite{ADN59} shows, by a perturbation argument the estimates are still valid for at least smooth variable coefficient operators and a general uniformly smooth domain $\Omega$.
% 原稿 2026/1/3、7-4、1/4
 We note that the complementing condition (LS) under (R) is a necessary and sufficient condition to have a solution in $L^2(0,\infty)$ the ordinary equation
 \begin{align*}
     A^0(i\xi',D_N) \hat{u}(\xi',x_n) &= 0
     \quad\text{for}\quad x_N > 0, \\
     B_j^0(i\xi',D_N) \hat{u}(\xi',0) &= \hat{g}_j(\xi'),
\end{align*}
where $\hat{}$ denotes the partial Fourier transform in the tangential variable in $x'$.
 Note that we only consider a solution which decays as $x_N\to\infty$.
 Results of \cite{ADN59}, \cite{ADN64} are very fundamental in modern theory of PDEs so there exist numerous extensions of their results.
 One important extension is for parabolic problems.

To handle a parabolic problem of the form
\begin{align}
    \partial_t u + (-1)^m A(D)u
    &= f &&\hspace{-6em}\text{in}\quad \Omega\times(0,T), \label{EPa1} \\
    B_j(D)u
    &= g_j, \quad
    j=1,\ldots,m &&\hspace{-6em}\text{on}\quad \partial\Omega\times(0,T), \label{EPa2} \\
    u &= a &&\hspace{-6em}\text{on}\quad \Omega\times\{0\}, \label{EPa3}
\end{align}
we need further conditions for $A$ and $B_j$'s.
 We say that an elliptic operator $A$ is parameter elliptic if there exists $\phi\in(0,\pi)$ such that
\[
    \left|\arg (-1)^m A^0(\xi)\right| < \phi
\]
for $\xi\in S^{N-1}$, where $\arg z$ denotes the argument of a complex number $z$, i.e.\ $z=|z|\exp(\arg z)$.
 Let $\phi_A$ denote the infimum of $\phi$.
 We say that $A$ is \emph{normally elliptic} if $\phi_A<\pi/2$.
 It turns out that normal ellipticity is equivalent to saying that
\[
    \mathcal{A}_\theta (D_x,D_t) = A(D_x)+ (-1)^m e^{i\theta} D_t^{2m}
\]
is elliptic for $\theta$ with $|\theta|<\pi-\phi_A$ with $\phi_A<\pi/2$.
 If we assume $N\ge2$, the operator $\mathcal{A}_\theta$ has at least three independent variables.
 Thus the root condition (R) for $\mathcal{A}_\theta$ is automatically fulfilled if $N\ge2$.
 In the case $N=1$, $A^0(\xi)=c\xi^{2m}$ with some $c\in\mathbb{C}\setminus\{0\}$ so $A^0(\tau)+(-1)^m\lambda$ has equal numbers of roots with positive imaginary part and with negative imaginary parts as a polynomial of $\tau$ where $|\arg\lambda|<\pi-\phi_A$.
 Thus, a similar root condition is fulfilled.

% 原稿 2026/1/3、7-4、2/4
There is another important notion of ellipticity.
 We say that $A$ is strongly elliptic of $\operatorname{Re}\left\{(-1)^m A^0(\xi)\right\}>0$ for $\xi\in\mathbb{R}^N$, $\xi\neq0$.
 It is known that a strong elliptic operator satisfies the root condition (R); see e.g.\ \cite[Theorem 5.4]{Ta}.
 Moreover, a strong elliptic operator is always normally elliptic.
 A class of normally elliptic operators was first introduced by S.~Agmon \cite{Ag} to study their resolvent problems.

The complementing condition (LS) should also be extended.
 It should be a complementing condition for $\mathcal{A}(D_x,D_t)$ with $\left\{B_j(D_x)\right\}_{j=1}^m$ in $\Omega\times(-\infty,\infty)$.
 Its explicit form is
\begin{enumerate}
 \item[(LSP)] Assume that $A$ is normally elliptic.
 There exists an angle $\phi>\phi_A$ which fulfills the following condition: for each tangent vector $\xi$ ($\neq0$)
 to $\partial\Omega$ and $\lambda\in\mathbb{C}\setminus\{0\}$ with $|\arg\lambda|<\pi-\phi$, let
\[
    \tau_1(\xi,\lambda),\ldots,\tau_m(\xi,\lambda)
\]
be the roots of polynomial $A^0(\xi+\tau\nu) +(-1)^m\lambda$ (in $\tau$) with positive imaginary part.
 (Such $m$ roots exist, as well as $m$ roots with negative imaginary part.)
 Then polynomials $\left\{B_j^0(\xi+\tau\nu)\right\}_{j=1}^m$ are linearly independent modulo $\prod_{j=1}^m\left(\tau-\tau_j(\xi,\lambda)\right)$.
\end{enumerate}
A typical result for the resolvent problem, which is originally due to S.~Agmon \cite[Theorem~2.1]{Ag}, is as follows; see also \cite[Lemma~5.7]{Ta}.
 The problem is
\begin{align}
    \left( \lambda+(-1)^m A(D) \right)u
    & = f &&\hspace{-8.5em}\text{in}\quad \Omega \label{ER1} \\
    B_j(D)u
    & = g_j, &&\hspace{-8.5em}j=1, \ldots, m \quad\text{on}\quad \partial\Omega. \label{ER2}
\end{align}
% 原稿 2026/1/3、7-4、3/4
\begin{prop} \label{PAg}
Assume that $A$ is normally elliptic and that $A$ and $B_j$ fulfill (LSP) with some angle $\phi\in\left(\phi_A,\pi/2\right)$.
 Assume that $p\in(1,\infty)$.
 Then for any $\delta>0$ there exist constants $M$ and $C$ such that
\begin{align*}
    \sum_{j=1}^{2m} |\lambda|^{1-j/(2m)} \|u\|_{W^{j,p}(\Omega)}
    &\le C\left[ \|f\|_{L^p(\Omega)}
    + \sum_{j=1}^m |\lambda|^{1-m_j/(2m)}
    \|g_j\|_{L^p(\Omega)}\right. \\
    &\left. + \sum_{j=0}^m \|g_j\|_{W^{2m-m_j,p}(\Omega)}\right]
    \quad\text{for all}\quad |\lambda|>M,\ 
    |\arg \lambda|<\pi- \phi-\delta
\end{align*}
for $u\in W^{2m,p}(\Omega)$ satisfying \eqref{ER1}, \eqref{ER2}, where $g_j\in W^{2m-m_j,p}(\Omega)$.
\end{prop}

This is a key estimate for the analyticity in $L^p(\Omega)$ of semigroups generated by the operator corresponding to $(-1)^mA$ with homogeneous boundary conditions, i.e.\ $g_j\equiv0$.
 A general $L^p$ theory by operator-theoretic method for parabolic problems is by now well developed based on vector-valued harmonic analysis since the space $X=L^p$ is the space of bounded Hilbert transform, i.e.\ $X$-valued Hilbert transform is bounded in $L^q(\mathbb{R},X)$ for $q\in(1,\infty)$.
  The reader is referred to a monograph \cite{PS} or an original paper by R.~Denk, M.~Hieber and J.~Pr\"uss \cite{DHP03} for \eqref{EPa1}--\eqref{EPa3}.

Let us only give a typical result.
 This is an easy consequence of \cite[Theorem 7.11]{DHP03}.
\begin{prop} \label{PPS}
Assume that $A$ is normally elliptic and that $A$ and $B_j$ fulfill (LSP) with some angle $\phi\in(\phi_A,\pi/2)$.
 For $p,q\in(1,\infty)$ let $u\in W^{1,q} \left(0,T;L^p(\Omega)\right)\cap L^q \left(0,T;W^{2m,p}(\Omega)\right)$ satisfy \eqref{EPa1}--\eqref{EPa3} with $g_j=0$, $a=0$.
 Then there exists a constant $C$ independent of $u$ and $f$ such that
\[
    \|u\|_{W^{1,q} \left(0,T;L^p(\Omega)\right)}
    + \|u\|_{L^q \left(0,T;W^{2m,p}(\Omega)\right)}
    \le C\|f\|_{L^q \left(0,T;L^p(\Omega)\right)}.
\]
\end{prop}
% 原稿 2026/1/3、7-4、4/4
In the case $p=q$, this is a classical result due to V.~A.~Solonnikov \cite{So65}.

For a parabolic problem, Schauder estimates are well developed by V.~A.~Solonnikov \cite{So65}.
 However, the abstract approach seems to be less popular, although there is an interesting work by A.~Lunardi, E.~Sinestrari, W.~von Wahl \cite{LSW}, where $\Omega$ is assumed to be bounded.

% 追加原稿 2026/1/23、2/5、ウ)
In \cite{So65}, V.~Solonnikov established Schauder estimates for a very general parabolic \emph{system} not necessarily of the form \eqref{EPa1}, where $A$ is an elliptic operator.
 There are several notions of parabolicity, such as in the sense of Petrovski\u{\i} \cite{Pe}, Shirota \cite{Shi1}, \cite{Shi2} and Douglis--Nirenberg \cite{DN}.
 The last one includes the former two notions.
% 追加原稿 2026/1/23、3/5、ウ)
 If the equation is of the form \eqref{EPa1}, then in \cite{So65}, $A(D)$ is assumed to be elliptic in the sense of Douglis--Nirenberg.
 Since in this paper we only discuss a single equation not a system, we only state a very special version of the Schauder estimates in \cite{So65}.
 Although in \cite{So65}, the domain $\Omega$ may not be the half space and coefficients of $A$ and $B$ are allowed to depend on $x$ and $t$, we only state the case that $\Omega$ is the half space and $A$ and $B$ are of constant coefficients.
 Let us state the Schauder estimate which is a very special form of \cite[Theorem 4.9]{So65}, where even the unique existence of a solution with Schauder estimates is stated.
\begin{thm} \label{TSol65}
Assume that $A$ is strongly elliptic.
 Assume moreover that $A$ and $B_j$ fulfill (LSP) with some $\phi\in(\phi_A,\pi/2)$.
 Assume $\mu\in(0,1)$.
 Let $u\in BC^{2m+\mu,(2m+\mu)/2m}\left(\Omega\times(0,T)\right)$ satisfy \eqref{EPa1}--\eqref{EPa3}.
 Then there exists a constant $C$ independent of $u$, $f$, $g_j$ and $a$ such that
\begin{gather*}
    \|u\|_{BC^{2m+\mu,(2m+\mu)/2m}\left(\bar{\Omega}\times(0,T)\right)}
    \le C \Biggl[ \|f\|_{BC^{\mu,\mu/2m}\left(\bar{\Omega}\times(0,T)\right)}
    + \sum_{j=1}^m \|g_j\|_{BC^{2m-m_j+\mu,(2m-m_j+\mu)/2m}\left(\partial\Omega\times(0,T)\right)} \\
    + \|a\|_{BC^{2m+\mu}(\bar{\Omega})} \Biggr].
\end{gather*}
\end{thm}

Note that the assumption on $A$ is a little bit stronger than that in Proposition~\ref{PPS} where $A$ is assumed to be only normally elliptic.

% 片山原稿 2026/1/25、P.23
In the rest of this section, we shall check whether our problems \eqref{EApA1}, \eqref{EApA2} and \eqref{EApA1}, \eqref{EApA3} satisfy the assumption of Theorem~\refeq{TSol65}.
 The operator $L^2$ is clearly strongly elliptic of order $4$, i.e., $m=2$, so one only has to check (LSP).
 Set
\begin{gather*}
A(\xi,\tau):=\left(|\xi|^2+\frac{1}{1+\gamma_0^2}\tau^2\right)^2,\\
B_1(\xi,\tau):=\frac{1}{1+\gamma_0^2}\tau,\quad B_2(\xi,\tau):=\frac{1}{1+\gamma_0^2}\left(|\xi|^2+\frac{1}{1+\gamma_0^2}\tau^2\right)\tau,\quad \xi\in\mathbb{R}^{N-1},\tau\in\mathbb{R}.
\end{gather*}
It is clear that $A^0=A$, $B_j^0=B_j$ ($j=1,2$) and $\phi_A=0$.
We take $\lambda\in\mathbb{C}\setminus\{0\}$ with $|\arg\lambda|<\pi-\phi$ for $\phi\in|0,\pi/2)$.
 Since $|\arg\lambda|<\pi-\phi$ implies 
\[
    \phi < \left|\arg(-\lambda)\right| \le \pi
\]
so that
\[
    \frac\phi2 < \left| \arg\pm(-\lambda)^{1/2}\right| < \pi-\frac\phi2,
\]
where we choose the branch of positive imaginary part of $(-\lambda)^{1/2}$.
 We are interested in the roots of polynomial $A^0(\xi,\tau)+\lambda$ for $\xi\in\mathbb{R}^{N-1}$.
 We may assume $\gamma_0=0$ by considering $\tau'=\tau/(1+\gamma_0^2)^{1/2}$ instead of $\tau$.
 It is easy to solve $A^0(\xi,\tau)+\lambda=0$.
 Its roots are
\[
    \tau = \pm \left( -|\xi|^2 \pm (-\lambda)^{1/2}\right)^{1/2}.
\]
We set
\[
    \tau_1 = \left( -|\xi|^2 + (-\lambda)^{1/2}\right)^{1/2}, \quad
    \tau_2 = \left( -|\xi|^2 - (-\lambda)^{1/2}\right)^{1/2}.
\]
Since $\left|\arg\pm(-\lambda)^{1/2}\right|>\phi/2$, $\tau_1$ and $\tau_2$ have positive imaginary part for any $\lambda\in\mathbb{C}\setminus\{0\}$ with $|\arg\lambda|<\pi-\phi$.
 Set
\[
    C(\tau) := (\tau-\tau_1) (\tau-\tau_2)
    = \tau^2 - \alpha\tau + \beta, \quad
    \alpha = \tau_1 + \tau_2, \quad
    \beta= \tau_1 \tau_2.
\]
Then we see that
\[
    B_2^0(\xi,\tau) = \left(|\xi|^2+\tau^2 \right)\tau
    = |\xi|^2 \tau+(\tau+\alpha)C(\tau)
    + \alpha^2\tau - \alpha\beta.
\]
Since $\tau_1$ and $\tau_2$ have positive imaginary part, we see that $\alpha\neq0$, $\beta\neq0$.
 Thus, $B_1^0=\tau$ and $B_2^0$ are clearly linear independent modulo $C(\tau)$ since $\alpha\beta\neq0$.
 Hence system \eqref{EApA1}, \eqref{EApA3} satisfies the condition (LSP).

Similarly, we observe that the system \eqref{EApA1}, \eqref{EApA2} also satisfies the condition (LSP).
 Thus the Schauder estimate is applicable to these linear systems.
 Namely,
\\
\noindent
\begin{prop}[See {\rm\cite[Section~15]{So65}}] \label{PSchauder}
  Let $k\in\mathbb{Z}_{\ge 0}$, $\mu\in(0,1)$, $T>0$ and
  \[
  f\in BC^{k+\mu,(k+\mu)/4}(\overline{\mathbb{R}^N_+}\times(0,T)),\quad a\in BC^{k+\mu}(\overline{\mathbb{R}^N_+}).
  \]
  Assume that $u\in BC^{4,1}(\overline{\mathbb{R}^N_+}\times(0,T))$ satisfies \eqref{EApA1}.
  Let $B\ge 0$ and assume further one of the following conditions:
  \begin{enumerate}[label={\rm(\roman*)}]
    \item $u$ satisfies the boundary condition \eqref{EApA2} with
    \[
    \|g\|_{BC^{k+4+\mu,(k+4+\mu)/4}(\partial\mathbb{R}^N_+\times(0,T))}+\|h\|_{BC^{k+2+\mu,(k+2+\mu)/4}(\partial\mathbb{R}^N_+\times(0,T))}\le B;
    \]
    \item $u$ satisfies the boundary condition \eqref{EApA3} with
        \[
    \|g\|_{BC^{k+3+\mu,(k+3+\mu)/4}(\partial\mathbb{R}^N_+\times(0,T))}+\|h\|_{BC^{k+1+\mu,(k+1+\mu)/4}(\partial\mathbb{R}^N_+\times(0,T))}\le B.
    \]
  \end{enumerate}
  Then $u$ belongs to $BC^{k+4+\mu,(k+4+\mu)/4}(\overline{\mathbb{R}^N_+}\times(0,\infty))$ and satisfies
  \[
  \|u\|_{BC^{k+4+\mu,(k+4+\mu)/4}(\overline{\mathbb{R}^N_+}\times(0,\infty))}\le C\left(\|f\|_{BC^{k+\mu,(k+\mu)/4}(\overline{\mathbb{R}^N_+}\times(0,\infty))}+\|a\|_{BC^{k+4+\mu}(\overline{\mathbb{R}^N_+})}+B\right).
  \]
\end{prop}
\begin{remark} \label{PPSchander}
  The strong ellipticity condition and condition~(LSP) are stable under perturbations small in H\"{o}lder norms.
  In particular, similar Schauder estimates are also applicable to the systems \eqref{ET1}--\eqref{ET3} and \eqref{ET1}, \eqref{EN2}, \eqref{EN3}.
\end{remark}

%%%%%%%%%%%%%%%%%%%%%%%%%%%%%%%%%%%%%%%%%%%%%%%%%
\appendix
\def\thesection{Appendix~B}% 要調査
\section{Graphical forms of differential operators on $\Gamma$} \label{SR} % Appendix B
\def\thesection{B} % 要調査

% 原稿 2025/12/8 A-1/3
In this appendix, we rewrite $\operatorname{div}_\Gamma$ and $\nabla_\Gamma$ when $\Gamma$ is given as the graph of a $C^1$ function $u=u(x)$ for $x\in\mathbb{R}^N$.
 Although these formulas are well-known, we give their proofs for the reader's convenience.

We first recall a definition of $\operatorname{div}_\Gamma$ and $\nabla_\Gamma$.
 Let $\bar{\nabla}$ denote the gradient in
\[
    \mathbb{R}^{N+1} = \left\{ (x,y) \mid
    x \in \mathbb{R}^N, y \in \mathbb{R} \right\},
\]
i.e.,
\[
    \bar{\nabla}= \left(\nabla,\frac{\partial}{\partial y}\right)
    = (D_1, \ldots, D_N, D_{N+1}).
\]
Let $\bar{X}$ be a $C^1$ vector field on $\Gamma$ of the form $\bar{X}=(X,Y)$ with $X\in C^1(\Gamma)^N$, $Y\in C^1(\Gamma)$.
 Let $f\in C^1(\Gamma)$.
 Let $\mathbf{n}$ be a unit normal vector field of $\Gamma$.
 Then, we define
\[
    \operatorname{div}_\Gamma \bar{X}
    := \operatorname{tr} \left( (I_{N+1} - \mathbf{n}\otimes\mathbf{n}) \bar{\nabla} \bar{X} \right), \quad
    \nabla_\Gamma f := (I_{N+1} - \mathbf{n}\otimes\mathbf{n}) \bar{\nabla} f.
\]
Here $\bar{X}$ and $f$ are extended in a tubular neighborhood $\mathcal{N}$ of $\Gamma$ so that it is $C^1$ in $\mathcal{N}$.
 As well known (see e.g.\ \cite{G06}), the values $\operatorname{div}_\Gamma\bar{X}$ and $\nabla_\Gamma f$ are independent of the choice of extensions.
 In the case when $\Gamma$ is the graph of a function, we may assume that $\bar{X}$ and $f$ are independent of $y$, i.e., $\partial_y\bar{X}=0$, $\partial_y f=0$.
 The operator $I_{N+1}-\mathbf{n}\otimes\mathbf{n}$ is the orthogonal projection to the tangent plane of $\Gamma$ at $\left(x,u(x)\right)$.
\begin{prop} \label{PA1}
Let $\Omega$ be a domain in $\mathbb{R}^N$.
 If $\Gamma$ is given as the graph of a function $u\in C^2(\Omega )$, then
\[
    (\operatorname{div}_\Gamma \bar{X}) \left(x,u(x)\right)
    = \omega^{-1} (x) \left( \operatorname{div}(\omega X) \right)(x)
    \quad\text{for}\quad x \in \Omega, \quad
    \omega = \left( 1+|\nabla u|^2 \right)^{1/2},
\]
provided that $\bar{X}$ is tangential to $\Gamma$, i.e., $\bar{X}\cdot\mathbf{n}=0$, where $\bar{X}=(X,Y)$ is a $C^1$ vector field on $\Gamma$ parametrized by $x\in\Omega$.
 Here $\operatorname{div}$ denotes the divergence in $\mathbb{R}^N$.
\end{prop}
\begin{proof}
Let $\mathbf{n}$ denote the upward unit normal vector field, i.e.,
\[
    \mathbf{n} = \frac{\mathbf{m}}{\omega}, \quad
    \mathbf{m} = (-\nabla u, 1).
\]
Since $\bar{X}\cdot\mathbf{n}=\sum_{i=1}^{N+1}X^i n_i$ with $X^{N+1}=Y$ is zero, we observe 
% 原稿 2025/12/8 A-2/3
that $D_j(\bar{X}\cdot\mathbf{m})=0$ so that
\[
    \sum_{i=1}^{N+1} m_i D_j X^i
    = \sum_{i=1}^{N+1} (-D_j m_i) X^i \quad
    \text{(}1\le j \le N+1\text{)},
\]
where $D_{N+1}=\partial/\partial y$ and $\mathbf{m}=(m_1,\ldots,m_{N+1})$.
 Since $m_{N+1}=1$ is a constant, we see that
\[
    \sum_{i=1}^{N+1} m_i D_j X^i
    = \sum_{i=1}^N (-D_j m_i) X^i.
\]
Since $\bar{X}$ is independent of $x_{N+1}=y$, we now conclude
% 原稿 2025/12/8 A-2/3
that
\begin{align*}
    \operatorname{tr}(\mathbf{n}\otimes\mathbf{n}\bar{\nabla}\bar{X})
    &= \sum_{1\le i,j\le N+1} \omega^{-2} m_i m_j D_j X^i
    = \omega^{-2} \sum_{1\le i,j\le N} m_j (-D_j m_i) X^i \\
    &= \omega^{-2} \sum_{1\le i,j\le N} m_j (-D_i m_j) X^i.
\end{align*}
Here we note that $m_j=-D_j u$ so that $D_im_j=D_j m_i$.
 Since $\nabla\omega=\omega^{-1}\sum_{j=1}^N m_j\nabla m_j$, we now observe that
\[
    \operatorname{tr}(\mathbf{n}\otimes\mathbf{n} \bar{\nabla} \bar{X})
    = -\omega^{-1} \nabla\omega\cdot X.
\]
Thus
\[
    \operatorname{div}_\Gamma\bar{X}=\operatorname{div}X+\omega^{-1} \nabla\omega\cdot X
    = \omega^{-1} \operatorname{div}(\omega X).
\]
\end{proof}

Let us now calculate $\nabla_\Gamma f$ when $\Gamma$ is given as the graph of a $C^1$ function $u$.
 We recall the $N\times N$ matrix
\[
    P(\nabla u) = I_N - \frac{\nabla u \otimes \nabla u}{\omega^2}.
\]
Then the operator $I_{N+1}-\mathbf{n}\otimes\mathbf{n}$ is of the form
\[
    I_{N+1} - \mathbf{n}\otimes\mathbf{n}
    = \omega^{-2}
    \begin{pmatrix}
        \omega^2 P(\nabla u) & \nabla u \\
        \nabla u^T & |\nabla u|^2
    \end{pmatrix},
\]
where $w^T$ denotes the transpose of a vector $w$.
 We thus obtain
\begin{prop} \label{PA2}
If $\Gamma$ is given as the graph of $u\in C^1(\Omega)$, then
\[
    (\nabla_\Gamma f) \left(x,u(x)\right)
    = \left( P(\nabla u)\nabla f,
   \frac{\nabla u \cdot \nabla f}{ \omega^2} \right) (x)
    \quad\text{for}\quad x \in \Omega,
\]
where $f$ is a $C^1$ function on $\Gamma$ parametrized by $x\in\Omega$.
\end{prop}
% 原稿 2025/12/8 A-3/3
Combining Propositions~\ref{PA1} and \ref{PA2}, we obtain a well-known formula for the Laplace-Beltrami operator.
\begin{prop} \label{PA3}
Assume that $\Gamma$ is given as the graph of $u\in C^2(\Omega)$, where $\Omega$ is a domain in $\mathbb{R}^N$.
 Let $f$ be a $C^2$ function on $\Gamma$ parametrized by $x\in\Omega$.
 Then
\[
    \Delta_\Gamma f
    := (\operatorname{div}_\Gamma \nabla_\Gamma f)\left(x,u(x)\right)
    = \left( \omega^{-1} \operatorname{div} \left( \omega P(\nabla u)\nabla f \right) \right)(x) \quad\text{for}\quad
    x \in \Omega.
\]
\end{prop}

If $\bar{X}$ is not tangential to $\Gamma$, then $\operatorname{div}_\Gamma\bar{X}$ is no longer equal to $\omega^{-1}\operatorname{div}(\omega X)$.
 A typical example is the mean curvature in the direction of $\mathbf{n}$, i.e.,
\[
    H = -\operatorname{div}_\Gamma \mathbf{n}.
\]
\begin{prop} \label{PA4}
If $\Gamma$ is given as the graph of $u\in C^2(\Omega)$, then
\[
    H \left(x,u(x)\right)
    = \left( \operatorname{div} \left(\frac{\nabla u}{\omega} \right)\right) (x) = \operatorname{tr} \left( \frac{P(\nabla u)\nabla^2 u}{\omega} \right)(x)
    \quad\text{for}\quad x \in \Omega.
\]
\end{prop}
\begin{proof}
If $\bar{X}=\mathbf{n}$ and $\mathbf{n}$ is extended to $\Omega\times\mathbb{R}$ so that $D_{N+1}\mathbf{n}=0$, then
\[
    \operatorname{tr}(\mathbf{n}\otimes\mathbf{n}\bar{X})
    = \sum_{1\le i,j\le N+1} n_i n_j D_j n_i
    = \sum_{1\le i,j\le N+1} n_j \frac12 D_j |n_i|^2 = 0
\]
since $|\mathbf{n}|=1$, where $\mathbf{n}=(n_1,\ldots,n_{N+1})$.
 Thus
\[
    H = -\operatorname{div}_\Gamma \mathbf{n}
    = -\operatorname{div}\mathbf{n}
    = -\operatorname{div}\left(-\frac{\nabla u}{\omega} \right)
    = \operatorname{div} \left( \frac{\nabla u}{\omega} \right).
\]
A direct calcuration shows that
\[
    \operatorname{div} \left( \frac{\nabla u}{\omega} \right)
    = \operatorname{tr} \left( \frac{P(\nabla u)\nabla^2 u}{\omega}  \right)
\]
so the proof is now complete.
\end{proof}

%%%%%%%%%%%%%%

\section*{Acknowledgements}
The work of the first author was partly supported by Japan Society for the Promotion of Science (JSPS) through grants KAKENHI Grant Numbers 24K00531, 24H00183 and by Arithmer Inc., Daikin Industries, Ltd.\ and Ebara Corporation through collaborative grants.
 The work of the second author was supported by Grant-in-Aid for JSPS Fellows DC1, Grant Number 23KJ0645.

%%%%%%%%%%%%%%%%%%%%%%%%%%%%%%%%%%%%%%%%%%%
%\bibliographystyle{alpha}
%\bibliography{bibdata}

\end{document}